\documentclass[11pt]{article}
\usepackage[english]{babel}
\usepackage{makeidx}
\makeindex

\usepackage[margin=0.85in]{geometry}
\usepackage{graphicx} 
\usepackage{amsthm}
\usepackage{mathtools}
\usepackage{calligra}
\usepackage{csquotes}
\usepackage{tensor}
\usepackage[thicklines]{cancel}
\usepackage{tcolorbox}
\usepackage{pstricks}
\usepackage{amsmath}
\usepackage{amssymb}
\usepackage{amsfonts}
\usepackage{faktor}
\usepackage{tikz-cd}
\usepackage{hyperref}
\usepackage{colortbl}
\usepackage{enumerate}
\usepackage{caption}
\usepackage{diagbox}
\usepackage{pdfpages}
\usepackage{tocloft}
\usepackage[page,toc,titletoc,title]{appendix}
\usepackage{CJKutf8}
\usepackage{mathrsfs}
\usetikzlibrary{patterns}

\DeclareCaptionType{diagram}[Diagram][]

\newtheorem{theorem}{Theorem}[subsection]
\newtheorem{remark}[theorem]{Remark}
\newtheorem{proposition}[theorem]{Proposition}

\newtheorem{definition}[theorem]{Definition}
\newtheorem{lemma}[theorem]{Lemma}
\newtheorem{example}[theorem]{Example}

\DeclareMathAlphabet{\mathcalligra}{T1}{calligra}{m}{n}
\DeclareFontShape{T1}{calligra}{m}{n}{<->s*[2.2]callig15}{}

\newcommand{\ie}{\emph{i.e.}} 

\newcommand{\prnt}[1]{\ensuremath{\left(#1\right)}} 
\newcommand{\chave}[1]{\ensuremath{\left\{#1\right\}}}  

\newcommand\numberthis{\addtocounter{equation}{1}\tag{\theequation}}

\providecommand{\keywords}[1]
{
  \textbf{\textit{Keywords---}} #1
}

\newcommand{\Eins}{{\mathbf 1}}

\newcommand{\ti}{\widetilde}
\newcommand{\ol}{\overline}

\newcommand{\RR}{\mathbb{R}}
\newcommand{\QQ}{\mathbb{Q}}
\newcommand{\ZZ}{\mathbb{Z}}
\newcommand{\HH}{\mathbb{H}}
\newcommand{\CC}{\mathbb{C}}
\newcommand{\PP}{\mathbb{P}}
\newcommand{\TT}{\mathbb{T}}
\newcommand{\UU}{\mathbb{U}}

\newcommand{\calM}{\mathcal{M}}

\newcommand{\calT}{\mathcal{T}}
\newcommand{\calA}{\mathcal{A}}
\newcommand{\calD}{\mathcal{D}}
\newcommand{\calL}{\mathcal{L}}

\newcommand{\calH}{\mathcal{H}}
\newcommand{\calK}{\mathcal{K}}

\newcommand{\calN}{\mathcal{N}}
\newcommand{\calS}{\mathcal{S}}

\newcommand{\End}{{\mathop{\mathrm{End}}}}

\newcommand{\CSpin}{{\mathop{\mathrm{CSpin}}}}
\newcommand{\KS}{\mathop{\mathrm{KS}}}

\newcommand{\GL}{\mathop{\mathrm{GL}}}

\newcommand{\SO}{{\mathop{\mathrm{SO}}}}
\newcommand{\Cl}{{\mathop{\mathrm{Cl}}}}
\newcommand{\Clp}{{\mathop{\mathrm{Cl}}}^+}
\newcommand{\T}{\mathop{\null\mathrm{T}}}
\newcommand{\Pic}{\mathop{\mathrm{Pic}}}

\newcommand{\AMT}{\mathcal{A}_{\mathcal{M},\mathcal{T}}}

\newcommand{\fro}{\mathfrak{o}}

\newcommand{\NS}{{\mathop{\mathrm{NS}}}}

\newcommand{\U}{\mathop{\null \mathrm{U}}}

\newcommand{\hotimes}{\mathop{\widehat{\otimes}}}
\newcommand{\wh}{\widehat}
\newcommand{\tAd}{{\widetilde{{\mathop{\mathrm{Ad}}}}}}
\newcommand{\Ad}{{{{\mathop{\mathrm{Ad}}}}}}

\newcommand{\Spin}{{\mathop{\mathrm{Spin}}}}
\newcommand{\Ogrp}{{\mathop{\null\mathrm{O}}}}
\newcommand{\MT}{{\mathop{\mathrm{MT}}}}

\newcommand{\diag}{{\mathop{\mathrm{diag}}}}
\newcommand{\DMT}{\mathcal{D}_{\mathcal{M},\mathcal{T}}}

\newcommand{\Nm}{{\mathop{\mathrm{Nm}}}}

\newcommand{\Srm}{K}
\newcommand{\Kum}{{\mathop{\mathrm{Kum}}}}

\newcommand{\im}{\mathop{\mathrm{Im}}}

\newcommand{\tr}{\mathop{\mathrm{tr}}}

\newcommand{\tiF}{\ti{F}}

\newcommand{\real}{\mathop{\mathrm{re}}}

\DeclareRobustCommand{\rchi}{{\mathpalette\irchi\relax}}
\newcommand{\irchi}[2]{\raisebox{\depth}{$#1\chi$}} 

\newcommand\textbox[1]{\parbox{0.75\textwidth}{\raggedright #1}}

\title{Kuga-Satake construction on families of K3 surfaces of Picard rank $14$}

\date{}
\author{Flora Poon\\Department of Mathematical Sciences, University of Bath \\North Rd, Claverton Down, Bath BA2 7AY, UK\footnote{The current address of the author is: National Center for Theoretical Sciences, No. 1, Sec. 4, Roosevelt Rd., Taipei City, Taiwan.\\The current email address of the author is: wkpoon@ncts.ntu.edu.tw.}}

\begin{document}

    \maketitle
\begin{abstract}

The isometry between the type IV$_6$ and the type II$_4$ hermitian symmetric domains suggests a possible relation between suitable moduli spaces of K3 surfaces of Picard rank $14$ and of polarised abelian $8$-folds with totally definite quaternion multiplication. 
We show how this isometry induces a geometrically meaningful map between such moduli spaces using the Kuga-Satake construction. Furthermore, we illustrate how the modular mapping can be realised for some specific families of K3 surfaces of Picard rank 14, which can be specialised to families of K3 surfaces of higher Picard rank.
\end{abstract}

\keywords{Kuga-Satake construction, coarse moduli spaces, K3 surfaces, abelian varieties.}

\subsection*{Data availability}
Data sharing not applicable to this article as no datasets were generated or analysed during the current study.

\subsection*{Conflict of interest}
There is no conflict of interest.
\section{Introduction}

Consider a lattice $P$ of signature $(1, r-1)$ with $1 \leq r \leq 20$ which primitively embeds into the K3 lattice~$\Lambda_{K3}$.
We will consider algebraic K3 surfaces over $\CC$ with quasi-polarisation by $P$, and we denote the moduli space of such K3 surfaces as $\calK_P$.
It is known \cite{d} that each irreducible component of~$\calK_P$ is a {\em locally symmetric variety} (LSV), which is a Riemannian manifold that is locally reflectionally symmetric around any point. 
Algebraically, they can be expressed \cite[Theorem VIII.7.1]{he} as biquotients in the form~$\Gamma \backslash G / K$, 
where $G$ is a connected simple Lie group, $K$ is a non-discrete maximal compact subgroup of $G$, and $\Gamma$ is an arithmetic subgroup of $G$. In this situation, $G/K$ is a {\em Hermitian symmetric domain} (HSD). 
Considering the different options for the Lie group~$G$ in this characterisation gives rise to a classification for both HSDs and LSVs: see \cite[Table~X.6.V]{he}. 
Under this classification, certain types of LSVs can be viewed as moduli spaces \cite[Section 3]{lo}:
the HSD~$G/K$ serves as the period domain, and the left translation action of the arithmetic subgroup $\Gamma<G$ on $G/K$ identifies period points that correspond to isomorphic polarised varieties. 
In particular, the period domain of~$\calK_P$, which is the set of weight two Hodge structures on the second integral cohomology group of a~$P$-polarised K3 surface, is the union of two copies of the HSD (of type IV$_{20-r}$ in the classification in \cite{he}) given by \[\SO^+(2, 20-r)/(\SO(2) \times \SO(20-r)).\] 
The lattice polarisation by $P$ determines an arithmetic subgroup $\Gamma(P) < \SO^+(2, 20-r)$ containing Hodge isometries that correspond to isomorphisms of $P$-polarised K3 surfaces.

We are especially interested in the type IV$_{20-r}$ series of LSVs, because for $r$ large, $\ie$ close to~$20$, the HSD overspace is isometric to the period domain of a different modular variety as Riemannian manifolds. 
For example \cite{ghs2}, the type IV$_4$ and type IV$_5$ HSDs ($r = 16$ and $r = 15$) coincide with the period domains of a modular variety of deformations of generalised Kummer varieties and of hyperk\"ahler manifolds of OG6 type respectively.
When $r = 14$,
the type IV$_6$ HSD is isometric to the type II$_4$ HSD  \cite[Exercise X.D.2(b)]{he}. 
The latter HSD, characterised as the Lie group quotient $\SO^*(8)/\U(4)$, can be identified with the set of weight one Hodge structures on the first integral cohomology group of a polarised abelian $8$-fold with totally definite quaternion multiplication, which is also the period domain of a modular variety $\AMT$ of polarised abelian $8$-folds whose PEL structures are controlled by certain attributes $\calM$ and $\calT$. 
In fact, $\AMT$ is a type II$_4$ LSV, $\ie$ it is an arithmetic quotient of $\SO^*(8)/\U(4)$ (see \cite[Section 9.7]{bl}). 

Let us denote the isometry between the HSDs in the case of $r = 14$ by $\tiF$. This map is studied in \cite{sa}.
Using results from \cite[Section 3.6]{sa}, we will show that any arithmetic subgroup $\Gamma(P) < \SO^+(2,6)$ can be mapped to an arithmetic subgroup $\Gamma^*(P)<\SO^*(8)$ (See the discussion following Theorem \ref{thm F}).
Moreover, $\tiF$ is equivariant with the natural conjugation actions of  $\Gamma(P)$ and $\Gamma^*(P)$, then $\tiF$ descends to a map $F$ from~$\calK_P$ to a certain moduli space $\AMT$ of abelian $8$-folds with totally definite quaternion multiplication. 
\begin{center}

\tikzset{every picture/.style={line width=0.75pt}} 

\begin{tikzpicture}[x=0.75pt,y=0.75pt,yscale=-1,xscale=1]

\draw    (378,53) -- (458.47,53.39) ;
\draw [shift={(460.47,53.4)}, rotate = 180.34] [color={rgb, 255:red, 0; green, 0; blue, 0 }  ][line width=0.75]    (10.93,-3.29) .. controls (6.95,-1.4) and (3.31,-0.3) .. (0,0) .. controls (3.31,0.3) and (6.95,1.4) .. (10.93,3.29)   ;
\draw    (286,66) -- (286.45,140.4) ;
\draw [shift={(286.47,142.4)}, rotate = 269.65] [color={rgb, 255:red, 0; green, 0; blue, 0 }  ][line width=0.75]    (10.93,-3.29) .. controls (6.95,-1.4) and (3.31,-0.3) .. (0,0) .. controls (3.31,0.3) and (6.95,1.4) .. (10.93,3.29)   ;
\draw    (512,68) -- (510.51,142.4) ;
\draw [shift={(510.47,144.4)}, rotate = 271.15] [color={rgb, 255:red, 0; green, 0; blue, 0 }  ][line width=0.75]    (10.93,-3.29) .. controls (6.95,-1.4) and (3.31,-0.3) .. (0,0) .. controls (3.31,0.3) and (6.95,1.4) .. (10.93,3.29)   ;
\draw    (306,156) -- (482.47,157.38) ;
\draw [shift={(484.47,157.4)}, rotate = 180.47] [color={rgb, 255:red, 0; green, 0; blue, 0 }  ][line width=0.75]    (10.93,-3.29) .. controls (6.95,-1.4) and (3.31,-0.3) .. (0,0) .. controls (3.31,0.3) and (6.95,1.4) .. (10.93,3.29)   ;

\draw (193,42.4) node [anchor=north west][inner sep=0.75pt]    {$\SO^{+}( 2,6) /(\SO( 2) \times \SO(6))$};
\draw (414,31.4) node [anchor=north west][inner sep=0.75pt]    {$\tilde{F}$};
\draw (466,42.4) node [anchor=north west][inner sep=0.75pt]    {$\SO^{*}( 8) /\U( 4)$};
\draw (289,92.4) node [anchor=north west][inner sep=0.75pt]    {$/\Gamma ( P)$};
\draw (514,85.4) node [anchor=north west][inner sep=0.75pt]    {$/\Gamma^{*}( P)$};
\draw (274,148.4) node [anchor=north west][inner sep=0.75pt]    {$\mathcal{K}_{P}$};
\draw (490,145.4) node [anchor=north west][inner sep=0.75pt]    {$\mathcal{A}_{\mathcal{M} ,\mathcal{T}}$};
\draw (393,137.4) node [anchor=north west][inner sep=0.75pt]    {$F$};

\end{tikzpicture}

    \captionof{diagram}{$\tiF$ descends.}\label{pic_descend}
\end{center}
The goal of this work is to realise the map $F$ as a modular mapping. More specifically, by identifying the domain and the target of $F$ as (an irreducible component of) the modular varieties $\calK_P$ and $\AMT$, we will explicitly describe how $F$ takes a K3 surface $X$ in $K_P$ to an abelian $8$-fold $A$ in $\AMT$. 
The {\em Global Torelli Theorem} which associates $X$ with its weight $2$ polarised Hodge structure on $H^2(X,\ZZ)$, and $A$ with its weight $1$ polarised Hodge structure on $H^1(A,\ZZ)$, reduces the problem to a purely lattice-theoretic one. 
This allows us to give a geometric interpretation of the map $F$ in terms of the {\em Kuga-Satake construction} \cite{ks}, a process that produces a weight $1$ Hodge structure from the Clifford algebra of a weight $2$ Hodge structure of K3 type.

The $r=14$ case is special. 
For slightly larger $r$, $\ie$ $r = 15$ and $16$, the question of finding an explicit modular mapping induced from the isometry of from a type IV$_4$ or type IV$_5$ HSD to the relevant period domain of a family of hyperk\"ahler manfiolds is premature: there is no known explicit family of deformations of generalised Kummer varieties or OG6 varieties.
Moreover, the case $r=14$ is the smallest $r$ such that there exists an isometry from a type IV$_{20-r}$ HSD to the period domain of a different modular variety. 
More importantly, the two period domains have the same dimension, so it is possible that the isometry induces an isomorphism $F$ of modular varieties upon choosing suitable arithmetic subgroup $\Gamma(P)$ in~$\SO^+(2,6)$.
If this induced map $F$ is holomorphic, then by \cite{kk} it extends to the Satake-Baily-Borel compactifications of $\calK_P$ and $\AMT$.

The layout of the paper is as follows:
In Section \ref{sect bg}, we will give the necessary set up for the construction of $F$, which includes 
a brief introduction of the moduli spaces $\calK_P$ and $\AMT$ (Section~\ref{subsect moduli}), the notion of Clifford algebras and some related concepts (Section \ref{subsect Clifford}), and the classical Kuga-Satake construction (Section \ref{subsect KS}). 
In Section \ref{sect F}, we will give an explicit description of a geometrically meaningful map $F$ from some suitable (see assumption (\ref{assumption2})) irreducible $\calK_P$ to $\AMT$ inspired by \cite{kstt}. 
We will also prove that the constructed map descends from the diffeomorphism $\tiF$ between the HSD overspaces.
In Section \ref{sect MAGMA}, we will describe some technicalities of the construction of $F$ by working on an example.
Finally in Section \ref{sect 18}, we will also observe some special behaviour of the map $F$ as we specialise our construction to subfamilies of $\calK_P$ which parametrise K3 surfaces with richer geometric properties.

\section*{Acknowledgement}
The author would like to thank her PhD supervisor Gregory Sankaran for his help and support through out the project. 
The author is also grateful to Bert van Geemen for the many useful discussions on Section \ref{sect MAGMA}, and to Alan Thompson for comments and advice.
This work also benefited from discussions with Andreas Malmendier, Adrian Clingher, Calla Tschanz and Alice Garbagnati. Special thanks go to the referee for pointing out some errors in the first version and for helping to make the exposition clearer.
The project is supported by EPSRC Research scheme EPSRC International Doctoral Scholars - IDS grant number EP/V520305/1.

\section{Background}\label{sect bg}
We will use the following notations:
If $L/K$ is a finite extension of fields, $R$ is a ring, $V$ is an $R$-module, and $K$ is an $R$-algebra, then we write $V_K$ for $V \otimes_R K$.

\subsection{Moduli spaces}\label{subsect moduli}
We will recall a minimal list of facts and properties of the moduli spaces $\calK_P$ and $\AMT$ for the construction of the map $F$, most of which are extracted from \cite[Section 2.2]{ae}, \cite{d}, \cite[Section 9]{bl} and \cite{sh}. 
Interested readers may also find in these sources more details of these moduli spaces and the varieties they parametrise.

\subsubsection{Moduli space of lattice polarised K3 surfaces}\label{subsubsect moduli K3}


We will give a brief description of the moduli space $\calK_P$ of $P$-polarised K3 surfaces.
First, we consider the case that the lattice $P$ satisfies the following condition:

\begin{equation}
    \textbox{all primitive embeddings $P \hookrightarrow \Lambda_{K3}$ lie in the same orbit of the isometry group  $\Ogrp(\Lambda_{K3})$ of the K3 lattice.}\tag{$\ast$}\label{assumption}
\end{equation}
In this case, the {\em transcendental lattice} $T$, which is the orthogonal complement of $i(P)$ in $\Lambda_{K3}$ for any primitive embedding
$i: P\hookrightarrow \Lambda_{K3}$, is defined uniquely up to isometry.
Let us consider the period domain $\calD_T$ of $\calK_P$, which is also the set of polarised weight two Hodge structures $T_\CC \simeq T^{2,0} \oplus T^{1,1} \oplus T^{0,2}$ on $T$.

\begin{proposition}\label{prop D_T}
    The period domain $\calD_T$ can be characterised in the following equivalent ways:
\begin{enumerate}
        \item {\em \cite[Remark 4.6]{vg}} as the set of group homomorphisms 
        \[\chave{h: \UU \longrightarrow \SO(T_\RR) \colon h(z)(t) = z^p\Bar{z}^qt \text{ for all }t \in T^{p,q} \subset T_\CC}\]
        where $\UU := \chave{z \in \CC^* \colon z\Bar{z} = 1}$ and $\SO(T_\RR)$ is the special orthogonal group.
        The set of group homomorphisms admits a natural group action of $\Ogrp(T_\RR) \simeq \Ogrp(2, 20-r)$ by conjugation. 
        $\ie$~for any $g \in \Ogrp(T_\RR)$,
        \[g\colon h \longmapsto h^g:=  ghg^{-1}.\]

        \item {\em \cite[Section 9]{dk}} as the set of projective lines $T^{2,0} \subset T_\CC$, which can be described as
        \[\chave{[l]\in \PP(T_\CC): l^2 = 0, l \cdot \Bar{l} > 0}.\]
        By writing $l = x+iy$ where $x, y \in T_\RR$ for any $[l]$, the above set is also the
        set of oriented planes~$\Pi := \langle x, y\rangle \subset T_\RR$ through the origin that are positive definite with respect to the restriction of the quadratic form of~$\Lambda_{K3}$ onto $T$.
        The latter set admits a natural group action of~$\Ogrp(T_\RR)$ by left multiplication.
        $\ie$~for any $g \in \Ogrp(T_\RR)$,
        \[g \colon \Pi \longmapsto g\Pi.\]
    \end{enumerate}
Moreover {\em  \cite[proof of Lemma 4.4]{vg}}, the two actions of $\Ogrp(T_\RR)\simeq \Ogrp(2, 20-r)$ are equivalent under the identification of the two characterisations of $\calD_T$.
\end{proposition}

From the second characterisation, it is apparent that $\calD_T$ has two connected components which consists of the positively oriented planes (correspond to the lines $[l]$)) and the negatively oriented planes (correspond to the lines $[\Bar{l}]$) respectively. We denote the former component, which is the image of the identity component of $\SO(T_\RR)$, by $\calD^+_T$.
In terms of the first characterisation, $\calD^+_T$ is also the set of group homomorphisms 
\[\chave{h \in \calD_T \colon h \text{ factors through } \SO^+(T_\RR)},\]
where $\SO^+(T_\RR)$ is the identity component of the group $\SO(T_\RR) \simeq \SO(2, 20-r)$.
By a generalisation of Witt's Theorem \cite[Remark 4.6]{vg},  the conjugation action of $\SO^+(T_\RR) < \Ogrp(T_\RR)$ on $\calD^+_T$ is transitive with stabiliser subgroup $\SO(2)\times \SO(20-r)$. 
Therefore, 
\begin{proposition}{\em \cite[Section 3]{dk}}\label{prop K_P}
    The HSD $\calD^+_T$ is the quotient space 
\[\SO^+(2,20-r)/\prnt{\SO(2)\times \SO(20-r)}.\]
Moreover, given the base $\calK_P$ of any complete family of $P$-polarised K3 surfaces with the assumption~(\ref{assumption}) satisfied, there exists an arithmetic subgroup~$\Gamma(P)$ of $\Ogrp(T_\RR)<\SO^+(2, 20-r)$ called the {\em monodromy group}, such that $\calK_P \simeq \Gamma(P) \backslash \calD_T$.
\end{proposition}
\begin{remark}\label{rmk D_T}
~\\
\vspace{-1.5em}
\begin{enumerate}
    \item From item $2$ in Proposition~\ref{prop D_T}, the dimension of a type IV$_{20-r}$ HSD, and thus of any of its arithmetic quotients, can be calculated as $20-r$.
    \item The Baily-Borel Theorem \cite[Section 4]{lo} says that an arithmetic quotient of a HSD is quasi-projective.
    If (\ref{assumption}) is satisfied and the hyperbolic lattice~$U$ is a summand in~$T$, then the monodromy group $\Gamma(P)$ swaps the two connected components of $\calD_T$, and $\calK_P$ is irreducible \cite[Proposition~5.6]{d}.
\end{enumerate}
\end{remark} 

If $P$ does not satisfy the assumption (\ref{assumption}), then 
\[\calK_P \simeq \bigcup_{l=1}^d \Gamma_l(P)\backslash\calD_T,\]
where each $\Gamma_l(P)$ is a monodromy group determined by an embedding $P \hookrightarrow \Lambda_{K3}$ in each $\Ogrp(\Lambda_{K3})$-orbit.

We say that a K3 surface~$X$ is {\em very general} in $\calK_P$
if the polarisation embedding $P \hookrightarrow \Pic(X)$ is surjective. 
In particular, all K3 surfaces that are not very general are contained in the union of countably many proper subvarieties of~$\calK_P$.

\subsubsection{Moduli space of abelian varieties with totally definite quaternion multiplication}\label{subsubsect moduli ab}

In this paper, we consider a $g$-dimensional abelian variety $A$ as a complex torus $\CC^g/\Lambda$ with a polarisation structure given by the first Chern class of an ample line bundle $L$ in $\NS(A)$, which we also denote by~$L$.
Equivalently, the complex torus can be replaced by a pair $(\TT, J)$. 
The first term is a real torus $\TT \simeq\Lambda^{\real}_\RR/\Lambda^{\real}$ determined by a lattice $\Lambda^{\real} \subset \RR^{2g}$.
The second item $J$ is a {\em complex structure}, which is a linear operator on  $\Lambda^{\real}_\RR$ satisfying $J^2 = -1$. The complex structure $J$ can be identified with a weight one Hodge structure on $\Lambda^{\real}_\CC$, which is the decomposition of $\Lambda^{\real}_\CC$ into the direct sum of the $(+i)$-eigenspace $(\Lambda^{\real})^{1,0}$ and the $(-i)$-eigenspace $(\Lambda^{\real})^{0,1}$ of $J$ respectively.
In particular, if we fix a pair~$(\Lambda^{\real}, J)$, then~$\Lambda$ is the image of $\Lambda^{\real}$
under the $\RR$-linear isomorphism
\begin{align}\label{eq mu}
    \mu: \Lambda^{\real}_\RR &\longrightarrow \CC^g \simeq (\Lambda^{\real})^{0,1} \\
    v &\longmapsto \frac{1}{2}\prnt{v-iJ(v)}.\nonumber
\end{align}
On the other hand, the choice of ample line bundle $L$ can be identified with an alternating form $E$ on~$\Lambda$ given by a matrix in the form
\[\begin{bmatrix}
        0 & D\\
        -D & 0
    \end{bmatrix}\]
where for a suitable choice of basis, $D = \diag(d_1, \cdots, d_g)$ with $d_i >0$ and $d_i|d_{i+1}$ for all $i$. The alternating form $E$ can be extended $\RR$-linearly to $\Lambda\otimes_\ZZ\RR = \CC^g$, and it defines a polarisation on the weight one Hodge structure $(\Lambda^{re}, J)$.
The $g$-by-$g$ matrix $D$ is  called the {\em polarisation type} of $A$. 

\begin{remark}\label{rmk E = polarisation}
    Conversely, an alternating form $E$ on $\CC^g$ represents the first Chern class of an ample line bundle if it satisfies an analogue of the Hodge-Riemann relations \cite[Theorem 2.1.6]{bl}.
    One of the conditions is that $E(\cdot, i\cdot) > 0$ for all $x \in \CC^g$.
\end{remark}

Next we describe endomorphism structures on an abelian variety~$A$.
\begin{definition}\label{def Ab End}{\em \cite[Section 9.1]{bl}}

    Let $(\mathbf{F},\rho)$ be a division ring of finite dimension over $\QQ$ and $\rho$ a positive anti-involution. 
    Let $\Phi$ be a representation of $\mathbf{F}$ by $g$-by-$g$ complex matrices
    $\Phi\colon \mathbf{F} \longrightarrow M_g(\CC)$.
    Then a $g$-dimensional abelian variety~$A$ has {\em endomorphism structure}  $(\mathbf{F}, \rho, \Phi)$ if there exists an embedding 
    $\iota:\mathbf{F} \hookrightarrow {\End}_\QQ(A) \subset M_g(\CC)$
    such that 
    \begin{enumerate}[(i)]
        \item $\Phi$ and $\iota$ are equivalent representations, and
        \item {\em (Rosati condition)} the Rosati involution on $\End_\QQ(A)$, which is an anti-involution induced by the polarisation of $A$, extends the anti-involution $\rho$ on $\mathbf{F}$ via $\iota$.
    \end{enumerate}
\end{definition}
We are interested in the case when   $\mathbf{F} \otimes_\QQ \RR \simeq \HH$, where $\HH$ is the Hamilton quaternions.
That is, when $\mathbf{F}$ is the quaternion algebra $(-a, -b)_\QQ$: it has  generators $I$, $J$ and $K$ such that $I^2 = -a \neq 0$, $J^2 = -b\neq 0$ and $K = IJ = -JI$.
In this case,~$g = 2m$ with~$m\in \ZZ_{>0}$ and the anti-involution~$\rho$ is given by the standard quaternionic conjugation. We will pick the representation $\Phi$ to be the representation as described in \cite[Section 2.2]{sh}
\begin{align*}
\Phi: \mathbf{F} &\longrightarrow M_g(\CC)\\
    x &\longmapsto \rchi(x) \otimes \Eins_m,
\end{align*}
where $\rchi$ is the representation of $\mathbf{F}$
   \begin{align*}
    \rchi\colon \mathbf{F} &\longrightarrow M_2(\CC)\\
    \alpha+\beta J &\longmapsto 
    \begin{pmatrix}
        \alpha_\CC & \sqrt{b}\cdot \beta_\CC\\
        -\sqrt{b}\cdot\ol{\beta_\CC} &\ol{\alpha_\CC}
    \end{pmatrix}.
\end{align*}
Here, $\otimes$ is the {\em Kronecker product of matrices}; $\alpha, \beta$ in the form of $r+sI$ with $r,s \in \QQ$; and $(r+sI)_\CC := r+s\sqrt{a}i$ is the corresponding complex number.
An abelian variety $A$ with such an endomorphism structure is said to admit a {\em totally definite quaternion multiplication}. 

We will now describe the moduli space $\AMT$ of abelian varieties of dimension $g=2m$ and polarisation type with totally definite quaternion multiplication.
For each member $A = \prnt{\CC^{2m}/\Lambda, E}$ in the family, one can associate a set of $m$ vectors $\chave{x_1, \cdots, x_m} \subset \CC^{2m}$ such that 
\begin{equation}
    \label{eq_xi}\Lambda_\QQ = \sum_{i=1}^m \Phi(\mathbf{F})x_i.
\end{equation}
Every member $A$ of $\AMT$ shares the same pair of attributes $(\calM, \calT)$.
The first attribute $\calM$ is a free~$\ZZ$-module of rank $4m$ in $\mathbf{F}^m$, such that when restricting Equation (\ref{eq_xi}) to the lattice $\Lambda$, we have 
\begin{equation}
    \label{eq_calM}\Lambda = \chave{\sum_{i=1}^m \Phi(a_i)x_i \colon (a_1, \cdots, a_m) \in \calM}.
\end{equation}
The second attribute is a non-degenerate matrix $\calT := (t_{ij}) \in M_m(\mathbf{F})$ which determines the alternating form $E$ on $\Lambda_\QQ$. 
In particular,  for all $x, y \in \Lambda_\QQ$, there exist suitable $a_i, b_j\in \mathbf{F}$ such that 
\begin{equation}
    \label{eq_ET}E(x, y) = E\prnt{\sum_{i=1}^m \Phi(a_i)x_i, \sum_{j=1}^m \Phi(b_j)x_j} = \tr\nolimits_{\mathbf{F}|_\QQ}\prnt{\sum_{i,j = 1}^m a_i t_{ij} b_j^\rho},
\end{equation}
where $\tr_{F|_\QQ}$ is the reduced trace over $\QQ$: the precise definition can be found in \cite[Section 5.5]{bl}.

Consider the period domain $\DMT$ of $\AMT$, which is the set of weight one Hodge structures on~$\Lambda^{\real}_\RR$. 
Like $\calK_P$, $\DMT$ can be expressed as a Lie group quotient.
Firstly, we have the following proposition involving the Lie group $\SO^*(2m)$, which can be considered \cite[Proposition 2.89]{harv} as the intersection of the indefinite unitary group $\U(m,m)$ and 
\[\GL{}_m(\HH) := \chave{M \in \GL{}_{2m}(\CC) : M\Phi(h) = \Phi(h)M \text{ for all } h \in \HH}.\]
\begin{proposition}\label{prop DMT}

    The period domain $\DMT$ 
    can be characterised in the following equivalent ways:
\begin{enumerate}
        \item {\em \cite[Section 4]{dk}} as the set of group homomorphisms 
        \[\chave{h: \UU \longrightarrow \SO^*(2m)<\GL(\Lambda_\RR): h(a+bi) = a\Eins_{4m}+bJ \text{, where } J \text{ is a complex structure on }\Lambda^{\real}_\RR}.\]
        Here, $\SO^*(2m)$ is viewed as a group of real $4m$-by-$4m$ matrices via
        \begin{align*}
           M_{2m}(\CC) &\longrightarrow M_{4m}(\RR)\\
            A + Bi &\longmapsto  \begin{bmatrix}A & B\\-B &A\end{bmatrix}
\text{, where }A, B \in M_{2m}(\RR),
        \end{align*}
        and it acts on the set of group homomorphisms by conjugation.
        Each representation $h$ can be recovered from the {\em Weil operator} $J$.

        \item {\em \cite[Section 9.5]{bl}} as the set of 
        normalised period matrices
        \[\chave{X = \begin{bmatrix}
        -Z & \Eins_m\\\Eins_m &\Bar{Z}
    \end{bmatrix}: Z \in \calH_m := \chave{Z \in M_m(\CC):= -Z = Z^t, 1-Z\Bar{Z}^t > 0}}. 
    \]
    Each normalised period matrix uniquely determines an $m$-vector $\chave{x_1, \cdots x_m}$ that satisfies (\ref{eq_xi}), thus a lattice $\Lambda < \CC^g$.
    The group $\SO^*(2m)$ acts on the set of period matrices by left multiplication.
    \end{enumerate}
Moreover, the two actions of $\SO^*(2m)$ are equivalent under the identification of the two sets.
\end{proposition}
Here we give a brief explanation for the last statement.
A bijection between the two sets is determined by the $\RR$-linear isomorphism $\mu$ given in Equation~(\ref{eq mu}).
The two actions of $\SO^*(2m)$ are compatible in the following sense:
for any $g \in \SO^*(2m)$,
the complex structure $gJg^{-1}$ is equivalent to a change of basis in $\RR^{2g}$ by left multiplication of~$g$, which corresponds to the isomorphism given by $g(1/2(\Eins - iJ))g^{-1}$.

By Witt's Theorem again, the conjugation action of $\SO^*(2m)$ on $\DMT$ is transitive with $\U(m)$ being the stabiliser group \cite[Section 9.7]{bl}. 
\begin{proposition}{\em \cite[Section 9.5]{bl}}\label{prop AMT}
    
    The period domain $\DMT$ is the quotient space \[\SO^*(2m)/\U(m).\]
    Moreover, any complete family $\AMT$ of abelian $2m$-folds with fixed attributes $\calM$, $\calT$ is isomorphic to the arithmetic quotient of $\DMT$ by the monodromy group~$\Gamma(\calM,\calT)<\SO^*(2m)$.
\end{proposition}


\begin{remark}\label{rmk DMT}
    From item $2$ in Proposition~\ref{prop DMT}, the dimension of $\AMT$ is~$ m(m-1)/2$. 
\end{remark} 

We say that an abelian variety $A$ is very general in $\AMT$ if $A$ is simple with $\End_\QQ(A) \otimes_\QQ \RR = \HH$.
\subsection{Clifford algebra}\label{subsect Clifford}
The following details about Clifford algebras and their significant subgroups are taken from \cite[Chapters~9-11]{harv}, \cite[Chapters 1.1-1.5]{lm} and \cite[Chapter 4.1]{hu}.

Let $R$ be $\ZZ$, $\QQ$, $\RR$ or $\CC$.
Let $V$ be an $R$-module of rank $n$ equipped with a non-degenerate symmetric bilinear form $b\colon V \times V \rightarrow R$ (thus a quadratic form $q(v) = b(v,v)$). Suppose $V$ is of signature $(n_+, n_-)$ where $n = n_+ + n_-$.

\begin{definition}{\em \cite[Section 4.1.1]{hu}}\label{def Clifford}

    The {\em Clifford algebra} $\Cl(V)$ over  $(V, q)$ is defined as 
\[\Cl(V) := \T(V)/I(V)\]
where $\T(V) := \sum_{k = 0}^\infty \bigotimes^kV$ is the {tensor algebra} over $V$, and $I(V) := \langle v \otimes v - q(v): v \in V \rangle$ is an ideal.
\end{definition}

There is a universal property for Clifford algebras.
\begin{lemma}[Fundamental lemma for Clifford algebras]\label{lem FLCA}{\em \cite[Lemma 9.7]{harv}}

    Let $A$ be an associative algebra with unit over $R$.
    Let $\varphi\colon V \rightarrow A$ be a linear map from $V$ into $A$.
    If for all $v \in V$ we have
    \[\varphi(v)\varphi(v) - q(v)\cdot \Eins_A = 0,\]
    then $\varphi$ has a unique extension to a homomorphism of algebras from $\Cl(V)$ to $A$.
\end{lemma}

Clifford algebras admit the following distinguished involutions:
\begin{definition}
~\\
\vspace{-1.5em}
    \begin{enumerate}
        \item {\em \cite[Equation 1.7]{lm}} The {\em canonical automorphism}
$( \cdot )^-\colon \Cl(V) \longrightarrow \Cl(V)$ is an involution defined by extending the isometry $v \mapsto -v$ on $V$
to an automorphism on $\Cl(V)$.

\item {\em \cite[Equation 1.15]{lm}} Consider the involution $( \cdot )^t\colon \T(V) \rightarrow \T(V)$ such that for any $v_1, v_2, \cdots, v_d \in~V$,
\[(v_1 \otimes v_2 \otimes \cdots \otimes v_d)^t = v_d \otimes \cdots \otimes v_2 \otimes v_1.\]
Since $( \cdot )^t$ sends the ideal $I(V)$ to itself, it descends to an involution on $\Cl(V)$ called the {\em transpose}, which we still denote by $( \cdot )^t$.
    \end{enumerate}
\end{definition}

The Clifford algebra is a $\ZZ_2$-graded algebra: since $I(V)$ only contains elements of even degree with respect to the natural $\ZZ$-grading of $T(V)$, the~$\ZZ$-grading descends to a $\ZZ_2$-grading for $\Cl(V)$. 
In particular, we have
\[\Cl(V) = \Cl^+(V) \oplus \Cl^-(V),\]
where $\Cl^+(V)$ is the {\em even part} of $\Cl(V)$ spanned by the classes of the even degree elements in $T(V)$, and $\Cl^-(V)$ is the {\em odd part} spanned by the classes of odd degree elements in $T(V)$.

\begin{remark}\label{rmk dim Clifford}
    By forgetting the Clifford multiplication, $\Cl(V)$ is isomorphic to the exterior algebra~$\bigwedge V$ as modules or vector spaces. Therefore, $2\dim\prnt{\Cl^+(V)} = \dim\prnt{\Cl(V)} = \dim\prnt{\bigwedge V} = 2^n$.
\end{remark}

Being $\ZZ_2$-graded algebras, Clifford algebras have a {\em graded tensor product} $\hotimes$.
Disregarding the multiplication, the graded tensor product of two graded algebras is the ordinary tensor product of graded modules \cite[Section IV.2]{la}.
In particular, the graded tensor product of two Clifford algebras is also a Clifford algebra with the usual Clifford multiplication.

\begin{lemma}[Gluing of Clifford algebras]\label{lem glueCl}{\em \cite[Lemma 1.7, Theorem 1.8]{la}}\index{Clifford algebra! gluing of $\sim$}

    Let $(V,q)$ and $(V',q')$ be two $R$-vector spaces/modules equipped with a quadratic form $q$ and $q'$ respectively.
    Then by the fundamental lemma for Clifford algebra, the linear map 
    \begin{align*}
        V \oplus V' &\longrightarrow \Cl(V)\hotimes \Cl(V')\\
        (v, v') &\longmapsto v \otimes \Eins + \Eins \otimes v'
    \end{align*}
     extends to a morphism of $\ZZ_2$-graded algebras
    \[f\colon \Cl(V \oplus V') \xrightarrow{~\simeq~} \Cl(V)\hotimes \Cl(V'),\]
    which is in fact an isomorphism.
\end{lemma}

A Clifford algebra with Clifford multiplication can be identified with a product of matrix algebras with the usual matrix multiplication.
First consider the case $R = \RR$. Let 
$\RR^{(n_+, n_-)}$ be the $\RR$-module of rank $n$ with the standard quadratic form of signature~$({n_+, n_-})$ given by
$v_1^2 + \cdots + v_{n_+}^2 - v_{n_++1}^2 - \cdots v_{n}^2$, 
and let \[\Cl(n_+,n_-):= \Cl(\RR^{(n_+, n_-)}).\]
There is a pattern for the product of matrix algebras that is isomorphic to~$\Cl(n_+, n_-)$ depending on $(n_+, n_-) \mod 8$: for details, see \cite[Theorem 11.3, Table~11.5]{harv}.
Furthermore, the identities \cite[Theorem 9.38, 9.43]{harv}
\begin{align}
    \label{eq clifford id1}\Clp(n_+ + 1, n_-) &\simeq  \Cl(n_+, n_-)\\
    \label{eq clifford id2}\Cl(n_+, n_-+1) &\simeq \Cl(n_-, n_+ +1)   
\end{align}
allow the even part of a Clifford algebra to be identified with a product of matrix algebras.
\begin{example}\label{ex Cl(2,6)}
    We will focus on the case $V \simeq \RR^{(2,6)}$. 
    Referencing \cite[Table 11.5]{harv}, we have
    \[\Cl^+(2,6) \simeq \Cl(2,5) \simeq M_4(\HH)\times M_4(\HH).\]
\end{example}

Now let us consider the case $R = \QQ$. Let $(V,q)$ be any $\QQ$-vector space.
The structure of $\Cl^+(V)$ partly depends \cite[Theorem 7.7]{vg} \cite[Section V.2]{la} on the determinant of $q$. 
We will discuss this with more detail at the beginning of Section \ref{subsect explicit F}.
When we further restrict to $R = \ZZ$, the image of $\Cl_\ZZ(n_+, n_-) := \Cl(\ZZ^{(n_+. n_-)})$ is contained in a maximal order in $\Cl_\QQ(n_+, n_-)$.

In fact, the above are identifications of $\ZZ_2$-graded algebras. 
The corresponding $\ZZ_2$-grading for the matrix algebras is called the {\em checkerboard grading}. In particular, for a ring $S$, the even and the odd parts of $M_d(S)$ consist of matrices such that no two adjacent entries, whether in the same row or in the same column, are both non-zero \cite[Section IV.2]{la}. 
By considering the Kronecker product of matrices~$\hotimes$, then $M_d(\mathscr{A}) \simeq M_d(S) \hotimes \mathscr{A}$ also admits a checkerboard grading. 
Its degree $0$ part and its degree $1$ part are respectively given by
    \[(\wh{M}_d(\mathscr{A}))_0 = \begin{bmatrix}
            \mathscr{A}_0 & \mathscr{A}_1 & \mathscr{A}_0 & \\
            \mathscr{A}_1 & \mathscr{A}_0 & \mathscr{A}_1 & \\
            \mathscr{A}_0 & \mathscr{A}_1 & \mathscr{A}_0 & \\[-5pt]
             & & &\hspace{-0.8em}\ddots
        \end{bmatrix}, ~
      (\wh{M}_d(\mathscr{A}))_1 = \begin{bmatrix}
            \mathscr{A}_1 & \mathscr{A}_0 & \mathscr{A}_1 & \\
            \mathscr{A}_0 & \mathscr{A}_1 & \mathscr{A}_0 & \\
            \mathscr{A}_1 & \mathscr{A}_0 & \mathscr{A}_1 & \\[-5pt]
             & & &\hspace{-0.8em}\ddots
        \end{bmatrix}.  
    \]

We end this subsection with the definition and properties of the spin group. We consider to $R = \RR$ and $V \simeq \RR^{(n_+,n_-)}$.

\begin{definition}{\em \cite[Section 10]{harv}}

    The {\em spin group} of $V$
    is defined as 
    \[\Spin(V) := \chave{x \in \Cl^*(V) \cap \Cl^+(V): x(x^-)^t=1, \tAd_x(v) \in V \text{ for all } v\in V},\]
    where $\Cl^*(V)$ is the multiplicative group of units of $\Cl(V)$, 
    and $\tAd$ is its twisted adjoint representation given by
     \begin{align*}
        \tAd\colon \Cl^*(V) &\longrightarrow \GL(\Cl(V))\\
        x &\longmapsto \tAd_x(\cdot):= \left [y \mapsto \prnt{x^-\cdot y\cdot x^{-1}}\right ].
    \end{align*}
        If $\Spin(V)$ is not connected, then its identity component is denoted by $\Spin^+(V)$. 
\end{definition}

\begin{remark}
    To avoid any confusion, we would like to emphasise that the $+$ decoration in $\Spin^+(V)$ is used in a similar sense as the $+$ decoration in $\SO^+(V)$, rather than as in $\Cl^+(V)$.
\end{remark}

We will only consider $\Spin(V) \subset \Cl^+(V) \simeq \Cl^+(n_+,n_-)$ when  $n_+ - n_- \equiv 0 \mod 4$. 
In this case, the spin group has two connected components. 
Its identity component $\Spin^+(V)$ has three distinct double cover group homomorphisms.
One of them is the restriction of the twisted adjoint representation, which is just the usual adjoint action $\Ad$ as $(\cdot)^-$ is trivial on $\Cl^+(V)$ by definition. 
The image of $\Ad$ is the group~$\SO^+(V)$.
On the other hand, $\Cl^+(V)$ is isomorphic to the product of two copies of a matrix algebra $M_d(W_+) \times M_d(W_-)$ where $W_\pm$ are complex vector spaces called the {\em spaces of half-spinors}. 
The image of $\Spin^+(V)$ is $\SO^*(W_+) \times \SO^*(W_-)$ and the projections $\varphi_\pm:~\Spin^+(V) \rightarrow \SO^*(W_\pm)$ down to each copy are double cover homomorphisms called the {\em half-spin representations} of $\Spin^+(V)$. 
\begin{center}

\tikzset{every picture/.style={line width=0.75pt}} 

\begin{tikzpicture}[x=0.75pt,y=0.75pt,yscale=-1,xscale=1]

\draw    (174,42) -- (161.89,86.71) ;
\draw [shift={(161,88.5)}, rotate = 285.31] [color={rgb, 255:red, 0; green, 0; blue, 0 }  ][line width=0.75]    (10.93,-3.29) .. controls (6.95,-1.4) and (3.31,-0.3) .. (0,0) .. controls (3.31,0.3) and (6.95,1.4) .. (10.93,3.29)   ;
\draw    (200,42.5) -- (400.53,86.75) ;
\draw [shift={(400.5,86.5)}, rotate = 190] [color={rgb, 255:red, 0; green, 0; blue, 0 }  ][line width=0.75]    (10.93,-3.29) .. controls (6.95,-1.4) and (3.31,-0.3) .. (0,0) .. controls (3.31,0.3) and (6.95,1.4) .. (10.93,3.29)   ;
\draw    (186,43.5) -- (320.89,86.21) ;
\draw [shift={(320,86.5)}, rotate = 200] [color={rgb, 255:red, 0; green, 0; blue, 0 }  ][line width=0.75]    (10.93,-3.29) .. controls (6.95,-1.4) and (3.31,-0.3) .. (0,0) .. controls (3.31,0.3) and (6.95,1.4) .. (10.93,3.29)   ;

\draw (152.5,18.4) node [anchor=north west][inner sep=0.75pt]    {$\Spin{}^{+}( V) \ \hookrightarrow \Cl{}^{+}( V) \simeq M_{d}( W_{+}) \ \times \ M_{d}( W_{-})$};
\draw (125,90.9) node [anchor=north west][inner sep=0.75pt]    {$\SO{}^{+}( V)$};
\draw (305,90.4) node [anchor=north west][inner sep=0.75pt]    {$\SO{}^{*}( W_{+})$};
\draw (390,89.9) node [anchor=north west][inner sep=0.75pt]    {$\SO{}^{*}( W_{-})$};
\draw (141.5,54.4) node [anchor=north west][inner sep=0.75pt]    {$\Ad$};
\draw (230.5,65.4) node [anchor=north west][inner sep=0.75pt]    {$\varphi _{+}$};
\draw (280.5,45.4) node [anchor=north west][inner sep=0.75pt]    {$\varphi _{-}$};
\draw (335, 65) node[rotate=90] {\scalebox{4}[1.5]{$\subset$}};
\draw (420, 65) node[rotate=90] {\scalebox{4}[1.5]{$\subset$}};

\end{tikzpicture}
    \captionof{diagram}{Three distinct double cover homomorphisms of $\Spin^+(V)$ when $n_+ - n_-$ is divisible by~$4$.}\label{pic_spin rep}
\end{center}

\subsection{Kuga-Satake construction}\label{subsect KS}

In this subsection, we explicitly construct a Kuga-Satake (KS) variety from a lattice polarised K3 surface. The main references are \cite[Section 4.2]{hu} and \cite{vg}.
The starting ingredient of the KS construction is a {\em Hodge structure of K3 type}, 
$\ie$ a weight two Hodge structure $V$ with $\dim V^{2,0} = 1$ and a quadratic form $q$ of signature $(\dim V - 2, 2)$ which is positive definite on $V^{1,1}$.

Let $\prnt{X, j\colon P \hookrightarrow \Pic(X)}$ be a K3 surface polarised by a rank $r$ lattice $P$, and let $T:= P^\perp_{\Lambda_{K3}}$ be its transcendental lattice which is of signature $(2, 20-r)$.
Note that $T$ has a Hodge structure of K3 type inherited from that of $H^2(X,\ZZ)$ with the intersection form, which is determined by choosing the $T^{2,0}$ part to be $H^{2,0}(X)\subset T_\CC$.
Using properties of Clifford algebras, we will construct from $(T,q)$, where $-q$ is the restriction of the intersection form, an abelian variety $\KS(T)$ called the KS variety.
\begin{remark}
    The identities (\ref{eq clifford id1}) and (\ref{eq clifford id2}) imply $\Cl^+(n_+, n_-) \simeq \Cl^+(n_-, n_+)$, so it does not matter whether we choose the quadratic form for $T$ to be $q$ or $-q$.
\end{remark}

By Remark \ref{rmk dim Clifford}, the Clifford algebra $\Cl(T)$ over $T$ is a lattice of rank $2^{22-r}$.
The quotient 
\begin{equation}
    \TT:=\Cl^+(T_\RR)/\Cl^+(T) \label{eq TT}
\end{equation}
is therefore a torus of real dimension $2^{21-r}$.
Moreover, $T^{2,0}$ determines a complex structure on $\TT$.
Pick a generator ~$\sigma=e_1 + ie_2$ of $T^{2,0}$ such that $e_1, e_2 \in T_\RR$ and $q(e_1) = 1$.
Since~$q(e_1+i e_2) = 0$, the vectors~$e_1$ and $e_2$ are orthonormal.
Set $J = e_1e_2$.  
One can check {\cite[Lemma~5.5]{vg}} that $J$ is an element in $\Spin^+(T_\RR)$ satisfying $J = e_1e_2 = -e_2e_1$ and~$J^2 = -1$. Furthermore \cite[Proposition 6.3.1]{vg}, $J$ is independent of the choice of the orthonormal basis $e_1, e_2$.
Under the homomorphism $\Ad:\Spin^+(T_\RR) \rightarrow \SO^+(T_\RR)$, the element $J$ then gives a complex structure on $\Cl^+(T_\RR)$ by left multiplication.

Finally, we give a construction of a polarisation\index{KS variety! polarisation of $\sim$} on the complex torus $(\TT, J)$.
Choose two orthogonal vectors $f_1, f_2 \in T$ with $q(f_i)>0$, and let $\alpha = f_1f_2$.
Consider the pairing $E$ with
\begin{align*}
    E \colon \Cl^+(T) \times \Cl^+(T) &\longrightarrow \ZZ\\
    (v,w) &\longmapsto \tr(\alpha (v^-)^tw) =  \tr(\alpha v^tw),\numberthis \label{eq E tr}
\end{align*}
where $\tr$ is the trace function for linear maps.
One can check \cite[Proposition 5.9]{vg} that the real extension $E_\RR$ of $E$ is an alternating form, and that the required conditions (Remark \ref{rmk E = polarisation}) for $E$ to be a polarisation of $(\TT,J)$ are satisfied upon choosing the correct sign of $\alpha$.

Therefore, $(\Cl^+(T_\RR)/\Cl^+(T), J, E)$ is an abelian variety of complex dimension $2^{20-r}$.
We call this abelian variety a KS variety, and denote it by $\KS(X)$ or $\KS(T)$. We have suppressed $\alpha$ in the notation because it is always clear what $\alpha$ is, or else the choice of polarisation class is unimportant.

\begin{remark}{\em \cite[Remark 4.2.3]{hu}} \label{rmk Cl+ = Cl-}

   Alternatively, one can define the KS variety from the odd part of the Clifford algebra $\Cl{}^-(V_\RR)$ instead of the even part: for any lattice $V$,
    fixing a vector $w$ in $V$ gives an isomorphism of $\RR$-vector space
    \begin{align*}
        \Cl{}^+(V_\RR) &\xrightarrow{\simeq~} \Cl{}^-(V_\RR)\\
        v &\longmapsto v\cdot w
    \end{align*}
    Moreover, the KS variety defined from $\Cl{}^+(V)$ is isogenous to the one defined from~$\Cl{}^-(V)$.
\end{remark}

\begin{remark}{\em \cite[Proposition 5.2.1, Theorem 5.3.2]{bl}}\label{rmk choice of alpha}

The choices of $\alpha$ in the KS construction can be described explicitly: given an abelian variety $A$, there is an
isomorphism of~$\QQ$-vector spaces 
    \[\NS_\QQ(A) \simeq \End^s_\QQ(A).\]
Here, the set $\End^s_\QQ(A)$ is the set of symmetric elements in $\End_\QQ(A)$, $\ie$ it consists of elements in~$\End_\QQ(A)$ which are stable under the Rosati involution $\rho$. 
\end{remark}

\section{Construction of modular mapping}\label{sect F}

\subsection{Explicit construction}\label{subsect explicit F}

In this section, we will construct explicitly the modular mapping $F: \calK_P \rightarrow \AMT$, where $\calK_P$ is a moduli space of K3 surfaces polarised by a rank $14$ lattice $P$, and $\AMT$ is a moduli space of abelian $8$-folds with totally definite quaternion multiplication associated to the pair $(\calM, \calT)$.
Note that the construction only works for $K_P$ satisfying the following condition:
\begin{equation}
    \textbox{the determinant of the quadratic form on the polarisation lattice  $P$ is a square in $\QQ^*$.}\tag{$\ast\ast$}\label{assumption2}
\end{equation}
Note that by \cite[Proposition 1.6.1]{n}, this is the same as requiring the determinant of the quadratic form on the transcendental lattice $T$ of $P$ to be a square in $\QQ^*$.
Assumption (\ref{assumption2}) is needed as our construction requires the the $\RR$-algebra isomorphism described in Example \ref{ex Cl(2,6)} to respect the rational structures on the $\RR$-algebras \cite[Theorem 7.7]{vg} \cite[Section V.2]{la}, $\ie$ we need
\[\Cl^+_\QQ(T_\QQ) \simeq M_4(\mathbf{F}) \times M_4(\mathbf{F}) \numberthis \label{eq Q Clifford iso}\]
where $\mathbf{F} \otimes_\QQ \RR = \HH$.

To begin with, let us apply the KS construction on the members of $\calK_P$. 
It is clear that the real torus~$\TT$ defined in Equation~(\ref{eq TT}) is the same for any $X \in \calK_P$ and $\alpha \in \Cl^+(T)$.
In fact, we may choose the same~$\alpha$ for every~$X \in \calK_P$ which gives rise to the same alternating form $E$.
Although the corresponding polarisation class in $\NS_\QQ(\KS(X))$ also depends on the complex structure $J$ of $\KS(X)$ as in Remark~\ref{rmk E = polarisation}, the polarisation type, which is discrete, remains constant as~$J$ varies in the family.
The construction, therefore, gives rise to a family of Kuga-Satake varieties $\KS(X)$ with $\alpha$ fixed, which embeds into the moduli space $\calA_{64}$ of polarised abelian $64$-folds.
The following theorem shows how we may derive an abelian $8$-fold $A_+$ with totally definite quaternion multiplication up to isogeny from the KS variety associated to a very general K3 surface in $\calK_P$.

\begin{theorem}\label{thm decomposeKS}
 For a very general K3 surface $X$ in the family $\calK_P$, there is a simple decomposition of $\KS(X)$ given by the isogeny
    \[\KS(X) \sim A_+^4 \times A_-^4,\]
    where $A_+$ and $A_-$ are non-isogenous simple abelian $8$-folds. Moreover, 
    \[\End_\QQ(A_+) \simeq \mathbf{F} \simeq \End_\QQ(A_-)\]
    with $\mathbf{F} \otimes_\QQ \RR = \HH$.
\end{theorem}

Before we give the proof, let us first recall the following consequence of the Poincar\'e's Complete Reducibility Theorem which will be useful for the rest of the paper.
    \begin{theorem}{\em \cite[Corollary 5.3.8]{bl}}\label{thm poincare}
    
         An abelian variety $A$ has a simple decomposition $A \sim \Pi_{i=1}^k A_i^{n_i}$ where $A_i$ is non-isogenous to $A_j$ for~$i \neq j$, if and only if 
    \[\End_\QQ(A) \simeq \Pi_{i=1}^k M_{n_i}\prnt{\End_\QQ(A_i)}.\]
    \end{theorem}

\begin{proof}[Proof of Theorem \ref{thm decomposeKS}]
    By the Global Torelli theorem and \cite[Corollary 3.6]{vg}, we have
    \[\End_\QQ(\KS(X)) \simeq \End_\MT(\Cl^+(T_\QQ)),\]
    where for any rational Hodge structure on $V$, $\End_{\MT}(V)$ consists of the vector space endomorphisms of $V$ that commute with the action of the Mumford-Tate group\index{Mumford-Tate group} $\MT(V)$.

    On the other hand, \cite[Lemma 6.5]{vg} we have
    \[\Cl^+(T_\QQ) \simeq \End_{\CSpin^+}(\Cl^+(T_\QQ)),\]
    where $\End_{\CSpin^+}(\Cl^+(T_\QQ))$ are the vector space endomorphisms of $\Cl^+(T_\QQ)$ that commute with the action of $\CSpin^+(T_\QQ)$, the identity component of the {\em special Clifford group} \cite[Section 10]{harv} that sits in the short exact sequence
    \[1 \longrightarrow \Spin(V) \longrightarrow \CSpin(V) \longrightarrow \RR^* \longrightarrow 1.\]
    If the Mumford-Tate group $\MT(\Cl^+(T_\QQ))$ is the special Clifford group $\CSpin^+(T_\QQ)$, then we are done by considering Equation~(\ref{eq Q Clifford iso}) and Theorem \ref{thm poincare}: $\KS(X) \sim A_+^4 \times A_-^4$ with $\dim_\CC(A_+) = \dim_\CC(A_-) = 8$.
    Indeed by a result of Zarhin \cite[Theorem 3.3.9, 6.4.9]{hu}, for a very general K3 surface $X$, we have
    \[\MT(T_\QQ) =  \MT(H^2(X,\QQ)) = \Ogrp(T_\QQ).\]
    Therefore by \cite[Proposition 6.3]{vg}, we have $\MT(\Cl^+(T_\QQ)) = \CSpin^+(T_\QQ)$.
\end{proof}

We claim that the isogeny \[f:=\KS(X)\xrightarrow[]{\sim} A_1 \times \cdots \times A_4 \times A_5 \times \cdots \times A_8,\] where $A_1, \cdots, A_4 \sim A_+$ and $A_5, \cdots, A_8 \sim A_-$, can be chosen in a compatible way for all~$X \in \calK_P$ so that $[X \mapsto \KS(X) \mapsto A_1]$ extends to the desired modular mapping $F$.
To prove the claim, we start by exploring all possibilities of the isogeny $f$ for $X$ very general.
\begin{theorem}\label{thm simple ab =s ymmetric idempotent}{\em \cite[Theorem 5.3.2]{bl}}

There is a bijection between the set of abelian subvarieties of an abelian variety $A$ and the set of symmetric idempotents in $\End_\QQ(A)$: 
let $\varepsilon$ be an idempotent in $\End^s_\QQ(A)$, $\ie$ $\varepsilon$ satisfies $\varepsilon^2 = \varepsilon$. 
If $d$ is the smallest positive integer such that~$d\varepsilon$ is in the order $\End(A)$, then under the above bijection, $\varepsilon$ corresponds to the abelian subvariety $\im(d\varepsilon)\subset A$.
\end{theorem}

By Theorem \ref{thm simple ab =s ymmetric idempotent}, we have 
\begin{equation}\label{eq A1}
    A_1 = \im(d_1\varepsilon_1)
\end{equation}
for some primitive $d_1\varepsilon_1 \in \End(\KS(X))$.
In fact, we can rewrite Equation (\ref{eq A1}) in terms of Clifford algebras only.
From the proof of Theorem \ref{thm decomposeKS}, it is clear that $\End_\QQ(\KS(X)) \simeq \Cl^+_\QQ(T)$.
Since $\TT$ in Equation~(\ref{eq TT}) is determined by the lattice $\Cl^+(T)$, the isomorphism restricts to an isomorphism of orders $\End(\KS(X)) \simeq \Cl^+(T)$.
In particular, if the real torus $\TT_1$ of $A_1$ is determined by the sublattice $\Lambda^{\real}_1$, $\ie$ if $\TT_1 \simeq (\Lambda^{\real}_1)_\RR/\Lambda^{\real}_1$, then Equation (\ref{eq A1}) becomes 
\begin{equation}\label{eq Lambda 1}
    \Lambda^{\real}_1 \simeq d_1\varepsilon_1 \cdot \Cl^+(T).
\end{equation}

We can explicitly compute an $\varepsilon_1 \in \Cl^+(T)$: let $\varphi$ be an isomorphism of $\QQ$-algebras as in Equation~(\ref{eq Q Clifford iso})
\begin{equation}\label{eq varphi}
    \varphi: \Cl^+(T_\QQ) \xrightarrow[]{~\simeq~} M_4(\mathbf{F}) \times M_4(\mathbf{F}).
\end{equation}
Let $E_{i,j} \in M_4(\mathbf{F})$ be the matrix
with $1$ at the $(i,j)$-th entry as the only non-zero entry. 
Then the elements in $M_4(\mathbf{F})\oplus M_4(\mathbf{F})$
\[(E_{1,1}, 0), \cdots, (E_{4,4}, 0), (0, E_{1,1}), \cdots, (0, E_{4,4})\]
are clearly symmetric idempotents of lowest possible rank, where here symmetric means stable under transpose of the matrix.
Considering Theorem \ref{thm poincare}, the element $(E_{1,1}, 0)$ acts on $A_1$. By choosing $\varepsilon_1$ to be the pull back of $\varepsilon_1$ via $\varphi$, we have the lattice $\Lambda^{\real}_1$ thus the real torus $\TT_1$ of $A_1$. 
Note that the complex structure~$J_1$ on $A_1$ is the restriction $J|_{\TT_1}$ of the complex structure $J$ of $\KS(X)$.
The polarisation $E$ of~$\KS(X)$ also restricts to a polarisation $E_i$ for $\TT_i$, and it has to be the unique one up to scalar multiples by Theorem~\ref{thm simple ab =s ymmetric idempotent}.
Similarly for each $i = 1, \cdots, 8$, we have the abelian $8$-fold  $A_i \sim (\TT_i, J_i, E_i)$, and $f$ as an isogeny of complex tori is given by $[p] \longmapsto \prnt{[d_1\varepsilon_1(p)], \cdots, [d_8\varepsilon_8(p)]}$.

It is clear that only the complex structure $J_1$, but not $\TT_1$ and $E_1$ of $A_1$, depends on the choice of the very general K3 surface $X \in \calK_P$ that we started with. 
Away from the very general members of $\calK_P$, the same choice of $\varphi$  still leads us to the same choice of the $\varepsilon_i$'s.
However, the resulting abelian~$8$-folds~$A_i$ are not very general, and they show exceptional behaviours. For example they may be no longer simple, or all of them belong to the same isogeny class.

This completes our proof for the following theorem.
\begin{theorem}\label{thm F}
    An isomorphism of algebras
    \[\varphi\colon \Cl^+(T_\QQ) \simeq M_4(\mathbf{F}) \oplus M_4(\mathbf{F})\]
    where $\mathbf{F}\otimes_\QQ \RR \simeq \HH$ induces a map $F$ from $\calK_P$ to a modular variety $\AMT$ of polarised abelian $8$-folds with totally definite quaternion multiplication (Diagram \ref{pic_decompKS}).
\end{theorem}

\begin{center}

\tikzset{every picture/.style={line width=0.75pt}} 

\begin{tikzpicture}[x=0.75pt,y=0.75pt,yscale=-1,xscale=1]

\draw    (171.4,56.4) -- (244.43,56.57) ;
\draw [shift={(246.43,56.57)}, rotate = 180.13] [color={rgb, 255:red, 0; green, 0; blue, 0 }  ][line width=0.75]    (10.93,-3.29) .. controls (6.95,-1.4) and (3.31,-0.3) .. (0,0) .. controls (3.31,0.3) and (6.95,1.4) .. (10.93,3.29)   ;
\draw [shift={(171.4,56.4)}, rotate = 180.13] [color={rgb, 255:red, 0; green, 0; blue, 0 }  ][line width=0.75]    (0,5.59) -- (0,-5.59)   ;
\draw    (319.4,56.4) -- (362.43,56.5) -- (392.43,56.57) ;
\draw [shift={(394.43,56.57)}, rotate = 180.13] [color={rgb, 255:red, 0; green, 0; blue, 0 }  ][line width=0.75]    (10.93,-3.29) .. controls (6.95,-1.4) and (3.31,-0.3) .. (0,0) .. controls (3.31,0.3) and (6.95,1.4) .. (10.93,3.29)   ;
\draw [shift={(319.4,56.4)}, rotate = 180.13] [color={rgb, 255:red, 0; green, 0; blue, 0 }  ][line width=0.75]    (0,5.59) -- (0,-5.59)   ;
\draw    (171.43,30.29) -- (244.43,30.29) ;
\draw [shift={(244.43,30.29)}, rotate = 180] [color={rgb, 255:red, 0; green, 0; blue, 0 }  ][line width=0.75]    (10.93,-3.29) .. controls (6.95,-1.4) and (3.31,-0.3) .. (0,0) .. controls (3.31,0.3) and (6.95,1.4) .. (10.93,3.29)   ;

\draw    (300.43,30.29) -- (370,30.29) ;
\draw [shift={(370,30.29)}, rotate = 180] [color={rgb, 255:red, 0; green, 0; blue, 0 }  ][line width=0.75]    (10.93,-3.29) .. controls (6.95,-1.4) and (3.31,-0.3) .. (0,0) .. controls (3.31,0.3) and (6.95,1.4) .. (10.93,3.29)   ;

\draw    (519,48) -- (567.93,90.97) ;
\draw [shift={(569.43,92.29)}, rotate = 221.29] [color={rgb, 255:red, 0; green, 0; blue, 0 }  ][line width=0.75]    (10.93,-3.29) .. controls (6.95,-1.4) and (3.31,-0.3) .. (0,0) .. controls (3.31,0.3) and (6.95,1.4) .. (10.93,3.29)   ;
\draw    (481,72) -- (529.93,114.97) ;
\draw [shift={(531.43,116.29)}, rotate = 221.29] [color={rgb, 255:red, 0; green, 0; blue, 0 }  ][line width=0.75]    (10.93,-3.29) .. controls (6.95,-1.4) and (3.31,-0.3) .. (0,0) .. controls (3.31,0.3) and (6.95,1.4) .. (10.93,3.29)   ;
\draw [shift={(481,72)}, rotate = 221.29] [color={rgb, 255:red, 0; green, 0; blue, 0 }  ][line width=0.75]    (0,5.59) -- (0,-5.59)   ;

\draw (123,19.4) node [anchor=north west][inner sep=0.75pt]    {$F: \calK_{P}$};
\draw (148,46) node [anchor=north west][inner sep=0.75pt]  [rotate=-271.02]  {$\in $};
\draw (270,46) node [anchor=north west][inner sep=0.75pt]  [rotate=-271.02]  {$\in $};
\draw (450,46) node [anchor=north west][inner sep=0.75pt]  [rotate=-271.02]  {$\in $};
\draw (550,110) node [anchor=north west][inner sep=0.75pt]  [rotate=-323]  {$\in $};
\draw (382,18.4) node [anchor=north west][inner sep=0.75pt]    {$\AMT \times \mathcal{A}_{\mathcal{M}_{2} ,\mathcal{T}_{2}} \times \cdots \times \mathcal{A}_{\mathcal{M}_{8} ,\mathcal{T}_{8}}$};
\draw (573.43,86.97) node [anchor=north west][inner sep=0.75pt]    {$\AMT$};
\draw (145,47.4) node [anchor=north west][inner sep=0.75pt]    {$X$};
\draw (410,47.4) node [anchor=north west][inner sep=0.75pt]    {$A_{1} \times A_{2} \cdots \times A_{8}$};
\draw (533.43,113.97) node [anchor=north west][inner sep=0.75pt]    {$A_{1}$};
\draw (557,56.4) node [anchor=north west][inner sep=0.75pt]    {$\pi _{1}$};
\draw (260,48.4) node [anchor=north west][inner sep=0.75pt]    {$\KS( X)$};
\draw (260,18.4) node [anchor=north west][inner sep=0.75pt]    {$\calA_{64}$};
\draw (354,62.4) node [anchor=north west][inner sep=0.75pt]    {$f$};
\draw (353,45.4) node [anchor=north west][inner sep=0.75pt]    {$\sim $};

\end{tikzpicture}
    \captionof{diagram}{The modular mapping F.}\label{pic_decompKS}
\end{center}

Next we will show that our modular mapping $F$ is indeed the one that fits in Diagram \ref{pic_descend}.
With reference to Diagram \ref{pic_spin rep},
there is a map 
\begin{equation}\label{eq tiF}
    \ti{F}: \calD^+_T \longrightarrow \calD_{\Cl^+(T)} \longrightarrow \DMT
\end{equation} 
that factors through the period domain $\calD_{\Cl^+(T)}$ of weight one Hodge structures on $\Cl^+(T)$.
Similar to Proposition \ref{prop DMT}(1), $\calD_{\Cl^+(T)}$ is the set of all representations $\ti{h}: \UU\rightarrow \Spin^+(T_\RR)$ such that  $J = \ti{h}(i) \in \Spin^+(T_\RR)\subset~\GL(\Cl^+(T_\RR))$ is a complex structure on $\Cl^+(T)$.

The first arrow in (\ref{eq tiF}), which is just the KS construction, corresponds \cite[Remark 4.2.1]{hu} to the lift of representations of $\UU$ with respect to the homomorphism $\Ad$ in Diagram \ref{pic_spin rep}.
    \begin{center}

\tikzset{every picture/.style={line width=0.75pt}} 

\begin{tikzpicture}[x=0.75pt,y=0.75pt,yscale=-1,xscale=1]

\draw    (268,96) -- (344.43,96.28) ;
\draw [shift={(346.43,96.29)}, rotate = 180.21] [color={rgb, 255:red, 0; green, 0; blue, 0 }  ][line width=0.75]    (10.93,-3.29) .. controls (6.95,-1.4) and (3.31,-0.3) .. (0,0) .. controls (3.31,0.3) and (6.95,1.4) .. (10.93,3.29)   ;
\draw    (264.43,81.29) -- (341.82,24.47) ;
\draw [shift={(343.43,23.29)}, rotate = 143.71] [color={rgb, 255:red, 0; green, 0; blue, 0 }  ][line width=0.75]    (10.93,-3.29) .. controls (6.95,-1.4) and (3.31,-0.3) .. (0,0) .. controls (3.31,0.3) and (6.95,1.4) .. (10.93,3.29)   ;
\draw    (381,32) -- (381.41,80.29) ;
\draw [shift={(381.43,82.29)}, rotate = 269.51] [color={rgb, 255:red, 0; green, 0; blue, 0 }  ][line width=0.75]    (10.93,-3.29) .. controls (6.95,-1.4) and (3.31,-0.3) .. (0,0) .. controls (3.31,0.3) and (6.95,1.4) .. (10.93,3.29)   ;

\draw (248,86.4) node [anchor=north west][inner sep=0.75pt]    {$\UU$};
\draw (349,88.4) node [anchor=north west][inner sep=0.75pt]    {$\SO^{+}(T_\RR)$};
\draw (347,11.4) node [anchor=north west][inner sep=0.75pt]    {$\Spin^+(T_\RR)$};
\draw (387,48.4) node [anchor=north west][inner sep=0.75pt]    {$\Ad$};
\draw (290,36.4) node [anchor=north west][inner sep=0.75pt]    {$\ti{h}$};
\draw (301,102.4) node [anchor=north west][inner sep=0.75pt]    {$h$};

\end{tikzpicture}
        \captionof{diagram}{Lift of representations.}\label{pic_lift}
    \end{center}
This lift of representations is not unique: suppose $h$ lifts to $\ti{h}$ and let $\ti{J} := \ti{h}(i) \in \Spin^+(T_\RR)$.
Then another representation $\ti{h'}: \UU \rightarrow \Spin^+(T_\RR)$ determined by $\ti{h'}(i) = -\ti{J}$ also descends to $h$ by $\Ad$.
However, only one of $\ti{J}$ and $-\ti{J}$ can meet the specific condition mentioned in Remark \ref{rmk E = polarisation} for the corresponding KS torus to be polarised.  
Therefore, there is a unique choice for the lift by further requiring $\ti{h}$ to be the complex structure of a polarised abelian variety with polarisation given by $\alpha$  (see Section \ref{subsect KS}), and the first arrow is injective.

The second arrow in (\ref{eq tiF}) is given by the composition of the half-spin representation $\varphi_+$:
it is clear that $\mathbf{F}^4$ where each copy of $M_4(\mathbf{F})$ acts on in (\ref{eq varphi}), when considered as a $\CC$-vector space, is the space of half-spinors.
We will abuse notation and denote the second arrow in (\ref{eq tiF}) by $\varphi_+$.

In fact, the map $\tiF := \varphi_+ \circ \Ad^{-1}$ is an holomorphic isometry mentioned in \cite[Section 3.6]{sa}, where it is called a representation of the type IV$_6$ HSD. This map $\tiF$ descends to a modular mapping $F$ naturally. To see this, consider $\Gamma$ an arithmetic subgroup of $\SO^+(2,6)$, we consider its inverse image in $\Spin^+(2,6)$ which is the extension of $\Gamma$ by $\ZZ/2\ZZ$, the kernel of the $\Ad$. 
Pushing forward this subgroup of $\Spin^+(2,6)$, we have a subgroup $\Gamma^*$ of $\SO^*(8)$.
This group $\Gamma^*$ is still an arithmetic subgroup of $\SO^*(8)$ as both $\Ad$ and $\phi^+$ are algebraic maps and thus the inverse image of $\Gamma$ is discrete.

It is easy to check that an orbit of $\SO^+(2,6)$ remains an orbit of $\SO^*(8)$ under $\tiF$:
Suppose we have an element $\gamma \in \Gamma$. Its preimage as a set is just $\{\pm \gamma\}$. The set is projected down to $\{\pm \gamma^*\}$ in $\SO^*(8)$. 
Let $[p_1] = [p_2] \in \calK_P$ with $\gamma \in \Gamma < \SO^+(2,6)$ such that $p_2 = \gamma p_1$.
Let $q_1 = \tiF(p_1)$ and $q_2 = \tiF(p_2)$. 
Then we know that either $q_2 = \gamma^* q_1$ or $q_2 = -\gamma^* q_1$. 
However, both options give $[q_1] = [q_2]$ in $\Gamma^*\backslash\DMT$, so $\tiF$ descends to the required map $F$.

\begin{remark}
    If assumption (\ref{assumption2}) is not satisfied, $\ie$ the determinant $d$ of the quadratic form on~$T_\QQ$ is not a square in $\QQ$, then \cite[Theorem 7.7]{vg} \cite[Section V.2]{la} we have $\Cl^+_\QQ(T_\QQ)\simeq M_4(\mathbf{F})$ where $\mathbf{F}$ is a quaternion algbra over $\QQ(\sqrt{d})$ instead. 
    By the same reasons as Theorem~\ref{thm decomposeKS}, we have $\KS(T) \sim A^4$ for a generic KS variety $\KS(T)$, where now $A$ is a simple abelian $16$-fold with $\End(A) \supseteq \mathbf{F}$.
    However, this does not descend from the holomorphic isometry $\tiF:= \varphi_+ \circ \Ad^{-1}$ studied by \cite{sa}.
\end{remark}

\begin{proposition}
    The modular mapping $F: \calK_P \rightarrow \AMT$ is a local diffeomorphism.
\end{proposition}
\begin{proof}
    By the Inverse Function Theorem, it is enough to show that $\tiF$ is bijective. 
    Denote by $(\calD_T^+)^{\KS}_+$ the subset of $(\calD_T^+)^{\KS}$ whose members are complex structures of a KS variety $\KS(T,\alpha)$.
    Then it suffices to prove that $(\varphi_+)_*$ in
    \begin{equation}
        \ti{F}\colon \calD^+_T \xrightarrow[]{~~\simeq~~} \prnt{\calD^+_T}^{\KS}_+ \xrightarrow[]{~~(\varphi_+)_*~~} \DMT
    \end{equation}
    is a bijection.
    Let
    \[K_\Spin := \Ad^{-1}\prnt{\SO(2)\times \SO(6)} \simeq \prnt{\Spin(2) \times \Spin(6)}/\chave{\pm (1,1)},\]
   which is a connected maximal compact subgroup of $\Spin^+(2,6)$.
   The left multiplication action of~$\Spin^+(2,6)$ on $(\calD_T^+)^{\KS}$ is transitive, so $K_\Spin$ is the stabiliser subgroup for the action, and
   \[\prnt{\calD^+_T}^{\KS}_+ \simeq \prnt{\Spin^+(2,6)/\chave{\pm 1}}/\Srm_\Spin \simeq \Spin^+(2,6)/\prnt{\Spin(2) \times \Spin(6)}.\]
    Note that as in Propositions~\ref{prop D_T} and~\ref{prop DMT}, the conjugation and left multiplication actions of $\Cl^+(2,6)$ on the set $(\calD_T^+)^{\KS}$ of Hodge structures in $\calD_T^+$ as described in Diagram \ref{pic_lift}, are equivalent. 
    Therefore the group $K_\Spin$ is also the stabiliser subgroup of the conjugation action of $\Spin^+(2,6)$, and so  \[(\varphi_+)_*\colon \Spin^+(2,6)/\Srm_\Spin \longrightarrow \SO^*(8)/\U(4)\] is surjective with kernel $\chave{1,\lambda} =: \ker(\varphi_+)$. 
   By considering Remark \ref{rmk E = polarisation} again, $(\varphi_+)_*$ is also injective.
\end{proof}

\begin{remark}
    Since $\tiF$ is holomorphic \cite{sa}, the modular mapping $F$ is also holomorphic, and therefore \cite{kk} $F$ extends to the Satake-Baily-Borel compactifications of $\calK_P$ and $\AMT$.
\end{remark}

\section{Computation of an example}\label{sect MAGMA}

In this section, we focus on the special family $\calK_P$ of K3 surfaces studied in \cite{cm2} where $P$ is an even, indefinite, 2-elementary polarising lattice
$P = U \oplus D_8(-1) \oplus D_4(-1)$.
By \cite[Proposition 1.14.4, Theorem 3.6.2]{n}, the lattice $P$ satisfies the assumption (\ref{assumption}), and
its transcendental lattice can be computed to be $T:= U \oplus U(2) \oplus D_4(-1)$.
The lattice also satisfies assumption (\ref{assumption2}) as its associated quadratic form has determinant $2^4$.
We will give an explicit construction of the map $F:\calK_P \rightarrow \AMT$ sending a K3 surface~$X$ to an abelian~$8$-fold~$A_1 = (\TT_1, J_1, E_1)$ as in Diagram \ref{pic_decompKS}.
All computations in this section were done using {\tt MAGMA}: details can be found in the author's PhD thesis~\cite{p}.

\subsection{Simple decomposition of KS variety}\label{subsect ex}

First, we fix the isomorphism $\varphi$ in Equation (\ref{eq varphi}). 
By Lemma \ref{lem glueCl}, it is enough to fix similar isomorphisms of $\QQ$-algebras for the indecomposible sublattices $U$, $U(2)$ and $D_4(-1)$ of $T$.
We will abuse notation and also call these isomorphisms $\varphi$.
We will apply the Fundamental Lemma for Clifford algebras.

First let us consider the lattice $U(n)$ for $n = 1,2$.
Let $\chave{f_1,f_2}$ be generators of the lattice $U(n)$ such that the associated symmetric bilinear form is given by the matrix 
\[M_{U(n)}:=
\begin{pmatrix}
0 & n \\
n & 0 \\
\end{pmatrix}.
\]
Then the linear map $\varphi$ determined by
\[
    \varphi(1) =  \Eins,~~
    \varphi(f_1) = \begin{pmatrix} 0 & 1 \\ 0 & 0 \\ \end{pmatrix},~~
    \varphi(f_2) = \begin{pmatrix} 0 & 0 \\ 2n & 0 \\ \end{pmatrix}.
\]
preserves the Clifford multiplication, thus extends to a $\ZZ$-algebra homomorphism ${\varphi:\Cl(U(n)) \rightarrow M_2(\QQ)}$.

Similarly, let $\chave{h_1, h_2, h_3, h_4}$ be the generators of the lattice $D_4(-1)$ such that the associated symmetric bilinear form is $-2M_\fro$, where
 \[M_\fro := \frac{1}{2}\begin{pmatrix}
        2&1&1&1\\
        1&2&0&0\\
        1&0&2&0\\
        1&0&0&2
    \end{pmatrix}.\]
Notice that the matrix $M_\fro$ is also the matrix of the symmetric bilinear form of the {\em Hurwitz integers} $\fro$ \cite[Section 5.1]{cs} defined as
\[ \fro := \ZZ\left\langle h:= \frac{1+i+j+k}{2}, i, j, k\right\rangle < \HH_\QQ := \QQ\langle i,j,k\rangle.\]
The Hurwitz integers form a maximal order in $\HH_\QQ = (-1, -1)_\QQ$
with a quadratic form $q$ given by the norm function $z \mapsto z \Bar{z}$.
Moreover \cite[Theorem 8.7]{rein}, $M_n(\fro)$ is a maximal order in $M_n(\HH_\QQ)$ for any integer~$n>0$.
This suggests that we can define a linear map $\varphi: D_4(-1) \rightarrow M_2(\fro)$ by setting
\[h_1 \mapsto \begin{pmatrix} 0 & h \\ -2\bar{h} & 0 \\ \end{pmatrix}, \quad h_2 \mapsto \begin{pmatrix} 0 & i \\ -2\bar{i} & 0 \\ \end{pmatrix}, \quad h_3 \mapsto \begin{pmatrix} 0 & j \\ -2\bar{j} & 0 \\ \end{pmatrix}, \quad h_4 \mapsto \begin{pmatrix} 0 & k \\ -2\bar{k} & 0 \\ \end{pmatrix}.\]
One can check that $\varphi$ satisfies $\varphi(z)^2 = q(z) \cdot \Eins$. Therefore, $\varphi$ extends uniquely to an algebra homomorphism $\varphi: \Cl(D_4(-1)) \rightarrow~M_2(\fro)$.

Lemma \ref{lem glueCl} then say that any two of the homomorphisms of graded algebras can be glued together by making use of their respective graded tensor products.
This gives us a homomorphism (which we still call $\varphi$) from the lattice $\Cl(T)$ to $M_8(\HH_\QQ)$, where $\varphi(x)$ for any $x \in \Cl^+(T)$ is in the form
\[\varphi(x) = \begin{pmatrix}
    m_{11} & 0 & 0 & m_{14} & 0 & m_{16} & m_{17} & 0\\
    0 & m_{22} & m_{23} & 0 & m_{25} & 0 & 0 & m_{28}\\
    0 & m_{32} & m_{33} & 0 & m_{35} & 0 & 0 & m_{38}\\
    m_{41} & 0 & 0 & m_{44} & 0 & m_{46} & m_{47} & 0\\
    0 & m_{52} & m_{53} & 0 & m_{55} & 0 & 0 & m_{58}\\
    m_{61} & 0 & 0 & m_{64} & 0 & m_{66} & m_{67} & 0\\
    m_{71} & 0 & 0 & m_{74} & 0 & m_{76} & m_{77} & 0\\
    0 & m_{82} & m_{83} & 0 & m_{85} & 0 & 0 & m_{88}
\end{pmatrix}  \in M_8(\HH_\QQ).\]
By extracting the two obvious $4$-by-$4$ blocks, $\varphi(\Cl^+(T))$ can be identified to a subset in $M_4(\HH_\QQ)~\oplus~M_4(\HH_\QQ)$:
\[\varphi(x) = \prnt{
\begin{pmatrix}
    m_{11} & m_{14} & m_{16} & m_{17}\\
    m_{41} & m_{44} & m_{46} & m_{47}\\
    m_{61} & m_{64} & m_{66} & m_{67}\\
    m_{71} & m_{74} & m_{76} & m_{77}
\end{pmatrix},
\begin{pmatrix}
    m_{22} & m_{23} & m_{25} & m_{28}\\
    m_{32} & m_{33} & m_{35} & m_{38}\\
    m_{52} & m_{53} & m_{55} & m_{58}\\
    m_{82} & m_{83} & m_{85} & m_{88}
\end{pmatrix}}.
\]
Extending linearly by $\QQ$, this gives us the required isomorphism $\varphi: \Cl^+(T_\QQ)\rightarrow M_4(\HH_\QQ)\oplus M_4(\HH_\QQ)$.

Let $\{f_1, f_2\}, \{f_3, f_4\}$ and $\{h_1, h_2, h_3, h_4\}$ be the sets of generators of  $U$, $U(2)$ and $D_4(-1)$ such that the matrices associated to the symmetric bilinear forms with respect to those generators are~$M_{U(1)}, M_{U(2)}$ and $-2M_{\fro}$ respectively. 
Let $\Eins$ be the identity element of any Clifford algebra.
Then $8E_{1,1},\cdots, 8E_{4,4}$ pull back via $\varphi: \Cl(U\oplus U(2)) \rightarrow M_4(\QQ)$ 
to
\[x_1 := f_3f_1f_2f_4,~~~~x_2 := 4f_1f_2 - x_1,~~~~x_3 := 2f_3f_4 - x_1,~~~~x_4 := 8\cdot\Eins - x_1 - x_2 - x_3,\]
which are pseudo-idempotents, $\ie$ primitive elements of $\Cl(U\oplus U(2))$ that are integral multiples of idempotents in $\Cl((U\oplus U(2))_\QQ)$.
Similarly, pulling back the diagonal matrices $\diag(4,0)$ and $\diag(0,4)$ via $\varphi\colon \Cl(D_4(-1)) \rightarrow M_2(\fro)$ give pseudo-idempotents
\[y_1 := 2-H,~~~~
    y_2 := 2+H.\]
By Lemma \ref{lem glueCl}, we then have the following eight pseudo-idempotents 
    \[\chave{32\varepsilon_1, \cdots, 32\varepsilon_8} = \chave{x_1\otimes y_1,~ x_2\otimes y_2,~ x_3\otimes y_2,~ x_4\otimes y_1,~ x_1\otimes y_2,~ x_2\otimes y_1,~ x_3\otimes y_1,~ x_4\otimes y_2}\]
in $\Cl(T)$ whose respective images under $\varphi$ are
    \[\chave{32E_{1,1}, 0), \cdots, (32E_{4,4}, 0), (0, 32E_{1,1}), \cdots, (0, 32E_{4,4})}.\]
    
     Since the sub-sublattices $U\oplus U(2)$ and $D_4(-1)$ are orthogonal to each other, the actions of the $x_j$'s commute with that of the $y_k$'s.
    Therefore, for the pseudo-idempotent $32\varepsilon_i = x_j\otimes y_k$, the lattice $\Lambda^{\real}_i$ is
    \[\Cl^+(T) \cdot 32\varepsilon_i = \chave{L\cdot K \in \Cl^+(T) : L\in \prnt{\Cl(U\oplus U(2))\cdot x_j}, ~K \in \prnt{\Cl(D_4(-1))\cdot y_k}}.\]
    Note that $\Cl(U\oplus U(-2))\cdot x_j = \ker(8\cdot\Eins - x_j)\subset \Cl(U\oplus U(-2))$.
    By applying the {\tt MAGMA} built-in function KernelMatrix, one can obtain the four primitive generators $L_1, \cdots, L_4\in \Cl(U\oplus U(2))\cdot x_j$, where two of them are in $\Cl^+(U\oplus U(2))$ and the other two are in $\Cl^-(U\oplus U(2))$.
    Similarly, one can obtain eight generators $K_1, \cdots, K_8$ for the lattice $\Cl(D_4(-1)) \cdot y_k$ where four of them are in $\Cl^+(D_4(-1))$, and the other four are in $\Cl^-(D_4(-1))$. 
    There are exactly $16$ pairs of $(s,w)$ such that the product $L_sK_w$ lies in $\Cl^+(T)$.
    These $16$ vectors form the $16$ generators of the lattice $\Lambda^{\real}_i\subset \RR^{16}$.
    Lastly, let us choose $\alpha := (f_1+f_2)\cdot (f_1 - f_2)$.
    As explained in Section \ref{sect F}, these are all that are needed to obtain the abelian $8$-fold $A_1$ when $i=1$ from a K3 surface $X\in \calK_P$.

    \subsection{A representation of the endomorphism algebra}\label{subsect rep End}
Ultimately, we would like to obtain the attributes $\chave{x_1, \cdots, x_4}, \calM, \calT$ associated to all $A_1$ obtained from the K3 surfaces in $\calK_P$ by the above means. 
    The key to solving Equations (\ref{eq_xi}), (\ref{eq_calM}) and (\ref{eq_ET}) is to obtain the representation $\Phi$ of the order $R := \End(A_1)<\mathbf{F} := \End_\QQ(A_1)\simeq \HH_\QQ$.
    We will compute a real representation $\Phi^{\real}$ out of $J_1$ and $\Lambda^{\real}_1$ with respect to the current basis of $\RR^{16}$ before transforming it to the  representation used in \cite{sh}.

     Before we begin the computations, let us recall the following facts about
     the endomorphism algebra~$\mathbf{F}$ of $A_1$.
     \begin{enumerate}
         \item The action of any element $f \in \mathbf{F}$ on $\CC^8 \simeq (\Lambda_1)_\RR$ is on the right as it has to commute with the left action of the complex structure~$J_1$. 
  However, under the representation $\Phi$  which identifies~$f \in \mathbf{F}$ to an element in $M_g(\CC)$, the matrix $\Phi(f) \in M_g(\CC) = \End(\CC^8)$ has the usual action on $\CC^8$ by left multiplication.

  \item Let $\varphi$ be as in Section \ref{subsect ex}. 
  The action of $\varphi(f)$ for any $f\in \mathbf{F}$  must commute with the natural action of $\varphi(\Cl^+(T))$ on the $\Cl^+(T)$-module $\Lambda^{\real}<(\HH_\QQ)^4$. 
  Therefore $\varphi(f)$ is a diagonal matrix.
     \end{enumerate}

 It is enough to define $\Phi$ on a set of four generators $\chave{r_1,\cdots, r_4}\subset R$, which can also be identified to a set $\{\ti{h}_1, \cdots, \ti{h}_4\} \subset \Cl^+(T)$ by the proof of Theorem \ref{thm decomposeKS}.
 Let us identify $h_i \in D_4(-1)$ with~$(1,1,h_i) \in~T \subset \Cl(T)$.
Define $c := 2h_1-h_2-h_3-h_4\in\Cl(T)$. 
Let
\[\ti{h}_1 := \Eins,~ 
     \ti{h}_2 := c \cdot h_1,~
     \ti{h}_3 :=c \cdot h_2,~
     \ti{h}_4 :=c \cdot h_3.\]
With {\tt MAGMA}, one can easily check that $\ti{h}_1, \cdots, \ti{h}_4$ together span a primitive lattice of rank $4$ in $\Cl^+(T)$, and that their images under $\varphi$ are diagonal matrices.
Let $N_i\in M_{16}(\ZZ)$ be matrices (with left multiplication on $\RR^{16} \simeq (\Lambda^{\real}_1)_\RR$) that correspond to the right actions of the elements $\ti{h}_i$ on $\RR^{16}$ with respect to the $16$ generators of $\Lambda^{\real}_1$ obtained in Section \ref{subsect ex}.
One can also check that the $N_i$'s span a primitive lattice in $M_{16}(\ZZ)$, and so we have the real representation $\Phi^{\real}$ of $F$ determined by $\ti{h}_i\mapsto N_i$.
In fact, each matrix in the image of $\Phi^{\real}$ contains many zero entries, which makes our computer computations 
very practical and efficient. 
\begin{lemma}\label{order of e_i}
    The real representation $\Phi^{\real}$  has image in the set of block diagonal matrices
    \[\chave{\diag(\calN_1, \cdots, \calN_4): \calN_j \in M_4(\ZZ)}\]
    with respect to a suitable order of the generators of the lattice $\Lambda^{\real}_1$ defining $A_1$.
\end{lemma}
\begin{proof}
    
    Recall in Section \ref{subsect ex}, each generator of the lattice $\Lambda^{\real}_1$ is a product of $L_s \in \Cl(U\oplus U(n))$ and~$K_w\in\Cl(D_4(-1))$. 
    In particular when fixing $s = s_0$, the set $\{L_{s_0}K_w: w = 1, \cdots, 4\}< \Lambda^{\real}$ generates a lattice of rank $4$ which corresponds to one of the four entries in the first column of $\varphi(\Cl^+(T))$.
    Since the action of the diagonal matrices $\varphi(\langle \ti{h}_1, \cdots, \ti{h}_4 \rangle)$ on the first column of $\varphi(\Cl^+(T))$ is equivalent to that of $\Phi^{\real}(\HH_\QQ)$ on $\RR^{16} \simeq \Lambda^{\real}_1 = \langle L_sK_w: s, w = 1, \cdots,4\rangle$, it is clear that with a suitable order of the generators $L_sK_w$, the image $\Phi^{\real}(\HH_\QQ)$ lies in the claimed set of block diagonal matrices.
 \end{proof}

With the isomorphism $\mu$ described in (\ref{eq mu}), the matrices $N_i$ can be identified with some complex matrices~$M_i\in M_8(\CC)$ such that they respectively represent the action of $\ti{h}_i$ on $\CC^8 \simeq (\Lambda_1)_\RR$ with respect to the eight~$+i$-eigenvectors of the complex structure $J_1$.
Then $\ti{h}_i\mapsto M_i$ determines a complex representation of~$\mathbf{F}$.
To be specific,
     one can choose an appropriate set of primitive generators $\chave{r_i}$ of $R$ satisfying the same multiplication rules as $\{\varphi(\ti{h}_i)\}$.
    Since the action of $\ti{h}_i\ti{h}_j$ from the right corresponds to the action of $M_jM_i$ from the left, the required algebra isomorphism $\Phi\colon F \rightarrow M_8(\CC)$ is then given by \[\iota(r_i) \longmapsto M_i,\] where $\iota$ is the anti-isomorphism of $\mathbf{F}$:
    \begin{align*}
    \iota\colon \HH_\QQ &\longrightarrow \HH_\QQ^{op}\\
    1, i, j &\longmapsto 1, i, j \text{ respectively,}\\
    k &\longmapsto -k.
\end{align*}

In our example, we have  \[\varphi(\ti{h}_1) = \Eins_2,~~
    \varphi(\ti{h}_2) = 
    \begin{pmatrix}
        -2\Bar{h}& 0\\
        0 & -2h
    \end{pmatrix},~~ 
    \varphi(\ti{h}_3) = 
    \begin{pmatrix}
        2i& 0\\
        0 & -2i
    \end{pmatrix},~~ 
    \varphi(\ti{h}_4) =
    \begin{pmatrix}
        2j& 0\\
        0 & -2j
    \end{pmatrix}.\]
    It suggests that the required complex representation $\Phi: \mathbf{F} \rightarrow M_8(\CC)$ can be defined by 
     \[1 \mapsto M_1 = \Eins_8, \quad (-1 + i + j - k) \mapsto  M_2, \quad 2i \mapsto M_3, \quad2j \mapsto M_4.\]

    The representation  can be transformed to the standard one $\Phi$ by a change of basis of $\CC^8$, $\ie$
there exists an~$8$-by-$8$ change of basis matrix $Q\in M_8(\CC)$, such that 
\[Q \cdot \Phi(r_i)= (\rchi(r_i) \otimes \Eins_4) \cdot Q \text{ for all }i = 1, \cdots, 4,\]
where $\rchi$ is the representation of $\mathbf{F}$ introduced in Section~\ref{subsubsect moduli ab}.
 To be specific, using the $\CC$-vector space isomorphism $(\cdot)^\sim: M_d(\CC) \rightarrow \CC^{d^2}$ described in \cite[p.252]{bl} which identifies a $d$-by-$d$ matrix $\{a_{ij}\}$ to a horizontal vector $a\ti{~} := (a_{11}, a_{12}, \cdots, a_{dd})$,  
one can compute matrices $A$ and $B$ such that for all $8$-by-$8$ matrix $M$ such that
\begin{align*}
    (M \cdot \Phi(r_i)){\ti{~}}&= M\ti{~} \cdot A, \\
    ((\rchi(r_i)\otimes \Eins_4) \cdot M){\ti{~}} &= M\ti{~} \cdot B.
\end{align*}
Then $Q$ is a (non-unique) $8$-by-$8$ non-singular matrix in the kernel space of $A-B$.

\subsection{Attributes}\label{subsect attributes}

In this section, we will compute the attributes $\chave{x_1, \cdots, x_4}, \calM, \calT$ associated to the image $A_1$ of $X\in \calK_P$.

We will start with $\chave{x_1, \cdots, x_4}$: first choose $\{(x^{\real})_1, \cdots, (x^{\real})_4\} \subset M_{16}(\ZZ)$ that satisfies the real version of Equation (\ref{eq_xi})
\begin{equation}\label{eq_x^re}
    (\Lambda^{\real}_1)_\QQ = \sum_{i=1}^4 \Phi^{\real}(\mathbf{F})(x^{\real})_i.
\end{equation}
We fix the order of the set of generators $\Lambda^{\real}_1$ such that the image $\Phi^{\real}(\HH_\QQ)$ are block diagonal matrices of $4$-by-$4$ blocks.
Then it is clear that the attributes $\{(x^{\real})_1, \cdots, (x^{\real})_4\}$ can be chosen to be $\{e_1, e_5, e_9, e_{13}\}$, where $e_j= (0, \cdots, 0, 1, 0, \cdots, 0)$ is the vector with $1$ as its $j^\text{th}$ entry.
The complex vectors $x_i$ that distinguish the members in $\AMT$ can then be obtained by multiplying the change of basis $Q$ obtained at the end of Section \ref{subsect rep End} to their images in the $+i$-eigenspace of the complex structure $J_1$ of $A_1$.
It can be checked that they do satisfy the original equation (\ref{eq_xi}), and it is clear that they are determined by $J_1$ as we vary $X \in \calK_P$.  

Next, we will compute the attributes $\calM$ of $\AMT$ from a real version of Equation (\ref{eq_calM}): 
\[\Lambda^{\real}_1 = \chave{\sum_{i=1}^4 \Phi^{\real}(a_i)(x^{\real})_i \colon (a_1, \cdots, a_4) \in \calM},\]
which does not depend on the complex structure of each member $A_1$.
In other words, we will identify $\Lambda^{\real}_1$ to a $\ZZ$-submodule $\calM$ of $\mathbf{F}^4$.  
Note that from Equation (\ref{eq_x^re}), we may decompose $\Lambda^{\real}_1$ into a direct sum of $\calL_i := \{\Phi^{\real}(a_i)(x^{\real})_i: (a_1,\cdots,a_4) \in \calM\}$ for $i = 1,\cdots, 4$, where $\calL_i$ is a $\ZZ$-module of rank $4$ that corresponds to the $i^\text{th}$ diagonal block of the elements in $\Phi^{\real}(F)$.
We will prove that for each $i$, the block~$\calL_i$ is isomorphic to an $\ZZ$-submodule~$\calM_i$ of $R =\langle r_1, \cdots, r_4\rangle$.
Consider the $\ZZ$-submodule~$R(x^{\real})_i <~\ZZ^4$ of $\calL_i$ generated by the vectors $\Phi^{\real}(r_1)e_1 = e_1, \cdots, \Phi^{\real}(r_4)e_{13}$ after removing unnecessary zeros.
Let~$(d_1, \cdots, d_4)$ with $d_j | d_{j+1}$ be the elementary divisors of the matrix
\[ \Big(e_1~|~\Phi^{\real}(r_2)e_5~|~\Phi^{\real}(r_3)e_9~|~\Phi^{\real}(r_4)e_{13}\Big) \in M_4(\RR),
\]
and let $d = d_4$.
Then $\calL_i$ is isomorphic to the $\ZZ$-module $d\calL_i<R(x^{\real})_i$. We can therefore obtain an~$R$-module $\calM_i$ by multiplying $d\calL_i$ by $(x^{\real})_i^{-1}$ on the right. 
Furthermore, $\calM_i$ is torsion free and is isomorphic to $\calL_i$:
\[
\begin{matrix}
   \calL_i &\xrightarrow{\hspace{2em} d \cdot \hspace{2em}} &d\calL_i &\xrightarrow{\hspace{1.7em} \cdot (x^{\real})_i^{-1} \hspace{1.7em}} \calM_i < R\\
      &   &\wedge & \\
      &   &R(x^{\real})_i &    
\end{matrix}
\]
This gives us \[\Lambda^{\real}_1 \simeq \bigoplus_{i=1}^{4}\calM_i < R^4 < \mathbf{F}^4.\]
We may even identify some of these $\calM_i$'s if they are isomorphic $R$-modules.
\begin{lemma} \label{lem iso R-mods}
Let $M$ and $N$ be $R$-submodules of $\mathbf{F}$ where $\mathbf{F}\otimes_\QQ \RR \simeq \HH$. Then $M$ and $N$ are isomorphic if and only if there exists $h \in \mathbf{F}$ such that~$N = Mh$. If $M$ and $N$ are also ideals of $R$, then the isomorphism preserves the number of minimal vectors (i.e. vectors of smallest norm) in the isomorphic ideals.
\end{lemma}
\begin{proof}
The reverse implication for the first statement is clear as $R$ is torsion free.
For the forward implication: suppose $f\colon M \rightarrow N$ is an $R$-module isomorphism. 
Fix any $m \in M$, so we have $Rm < M$. Similar to the above, by considering the elementary divisors of $Rm$ in $M$, we can find an integer $d$ such that any $x \in M$ may be written as $x = rm/d$ for some $r \in R$. Now 
\[f(x) = \frac{rf(m)}{d} = \frac{rm \cdot m^{-1} \cdot f(m)}{d} = x(m^{-1} \cdot f(m))\]
where $m^{-1} \cdot f(m) \in \mathbf{F}$.

The norm in $R$ is defined as $\Nm(r) = r\Bar{r}$ for all $r \in R$. So if $x \in M$ is a minimal vector, then $xh \in N$ is a minimal vector with norm $\Nm(x)\Nm(h)$.
\end{proof}

Recall that $R$ is the order spanned by $2\fro$ and an extra generator $1$ (see Section~\ref{subsect rep End}). 
Using the ShortestVectors function in {\tt MAGMA},
 it can be shown, up to reordering the index $i$ for the modules $\calM_i$, that $\calM_1$ and $\calM_2$ have $6$ minimal vectors, while $\calM_3$ and~$\calM_4$ have $12$.
On the other hand, the $R$-modules in $\HH_\QQ$
\begin{align*}
    I_6 &:= \langle h+i, h+j, i-j, k \rangle\\
    I_{12} &:= \fro = \langle h, i, j, k\rangle
\end{align*}
have $6$ and $12$ minimal vectors respectively.
One can show that $\calM_1 \simeq \calM_2 \simeq I_6$, and $\calM_3 \simeq \calM_4 \simeq I_{12}$ by brute force.
To be specific, suppose $\calM_i$ and $I_n$ have minimal vectors $u_1, \cdots, u_n$ and $v_1, \cdots, v_n$ respectively.
Then by Lemma \ref{lem iso R-mods}, $M_i \simeq I_n$ if and only if there exists $h_{kl}:= u_k^{-1}v_l \in \HH_\QQ$ for some indices $k, l \in 1, \cdots, n$ such that right multiplication by $h_{kl}$ is a bijection between the sets of generators or minimal vectors of $\calM_i$ and $I_n$.
Summarising, $\Lambda_1$ is isomorphic to the $\ZZ$-module in $\mathbf{F}^4$
\[\calM = I_6 \oplus I_6 \oplus I_{12} \oplus I_{12}.\]

Finally, we move on to the calculation of the last attribute $\calT= \chave{t_{ij}}$ that satisfies Equation (\ref{eq_ET}).
We consider the alternating form $E$ as the pairing on $(\Lambda^{\real}_1)_\RR \simeq \RR^{16}$ given in Equation~(\ref{eq E tr}) choosing~$\alpha$ to be $(f_1+f_2)(f_3+f_4) \in \Cl^+(T)$.
Also, let~$S_{i}$ be the $4$-by-$4$ matrix representing the right multiplication of $h_{kl}$ on $\calM_i$ that gives the isomorphism~$\calM_i\simeq I_n$. 
Then $S := \diag(S_1, \cdots, S_4)$ is a matrix taking any element in $\Lambda^{\real}_1$ with respect to the generators $e_j$'s to its image in $\calM := \calM_1 \oplus \cdots \oplus \calM_4$, where each~$\calM_i$ is with respect to the basis $\{r_1, \cdots, r_4\}$ of $R$.
By identifying each $\calM_i$ with a submodule in $\HH$, there is a $4$-by-$16$ matrix $S'$ over $\HH$ such that $S'(e_j) < \HH^4$ represents the same element as $S(e_j)$.
Then~$\calT$ is the unique $4$-by-$4$ matrix such that
\[(S'_h)^t\calT \ol{S'_l} = ({M_E})_{h,l},\]
where $S'_h$ and $S'_l$ are the $h$-th and the $l$-th columns of $S'$, and $M_E$ is the matrix of $E$ with respect to the basis $\chave{e_1,\cdots, e_{16}}$ of $\Lambda^{\real}_1$.
In fact, $M_E$ is in the form
\[M_E = \begin{bmatrix}
    0 &\ast &0 & 0\\
    \ast &0 &0 &0\\
    0 &0 &0 &\ast\\
    0 &0 &\ast &0
\end{bmatrix} \in M_{16}(\RR),\]
where each asterisk represents a non-zero $4$-by-$4$ block.
This implies that the matrix $\calT$ only has four non-zero entries: $t_{1,2}, t_{2,1}, t_{3,4}$ and $t_{4,3}$.
To solve for any of these non-zero entries, say $t_{1,2}$, it is enough to consider the four equations
\[S'_{1,1}\cdot t_{1,2} \cdot S'_{2, 4+k} = (M_E)_{1,4+k} ~\text{where}~ k = 1, \cdots, 4.\]
The calculation gives
\[\calT = \begin{pmatrix} 
0 & 256 & 0 & 0 \\
-256 & 0 & 0 & 0\\
0 & 0 & 0 & -512\\
0 & 0 & 512 & 0
\end{pmatrix}.\]
Thus the elementary divisors of the polarisation of $A_1$ are $1$ and $2$ both with multiplicity $4$.
One can also check that the matrix $\calT$ is the same for all $\Lambda^{\real}_i$ for $i = 1, \cdots, 8$ up to swapping the two copies of $I_6$ (and/or the two copies of $I_{12}$) in $\calM = I_6 \oplus I_6 \oplus I_{12} \oplus I_{12}$.

Furthermore, one can compute the image of $\tiF: \calD^+_T \rightarrow \DMT \simeq \calH_4$ following \cite[Section 2]{sh}.
We have shown that the complex structure $J_1$ of $A_1$ gives the attribute $\chave{x_1, \cdots, x_4}$, which can in fact be standardised by associating it to a period matrix $X \in M_g(\CC)$. 
Write each vector $x_i$ in the form~$\prnt{u_i~|~v_i}^t$
where $u_i, v_i \in M_{4\times 1}(\CC)$, and put $U = \prnt{u_1, \cdots, u_4},~ V = \prnt{v_1, \cdots, v_4}$. 
Define a matrix
\[X := \begin{bmatrix}
    U & V \\ \ol{V} & -\ol{U}
\end{bmatrix}.\]
Upon choosing a suitable basis of $\mathbf{F}_\RR^4$ such that $\calT^{-1}$ is given by $\sqrt{-1}\cdot \Eins_4$ with respect to $\calM$,
or equivalently the complex matrix $\sqrt{-1}\rchi(\calT)^{-1}$ is in the form
$\diag(-\Eins_4, \Eins_4)$, then the $4$-by-$4$ complex matrix $Z := -V^{-1}U$ satisfies
$Z^t = -Z$ and $1-Z\ol{Z}^t > 0$. 
Furthermore by change of basis of $\CC^8$, that is by the left multiplication action of $\GL_8(\CC)$, we can assume that $V = \Eins_4$, and the period matrix $X$ is in the standardised normalised form\index{period matrix! standardised normalised}  
    \[\begin{bmatrix}
        -Z & \Eins_4\\\Eins_4 &\Bar{Z}
    \end{bmatrix}
    \]
which is unique to the attribute $\chave{x_1, \cdots, x_4}$.
As an example, we compute the image under 
$\ti{F}$ for the point $\omega := [\langle (f_1+f_2)/\sqrt{2} - i(f_3+f_4)/2 \rangle_\CC] \in \calD^+_T$.
By the above steps,
\[\ti{F}(\omega) = \begin{pmatrix}
    0 &a &0 &0\\
    -a &0 &0 &0\\
    0 &0 &0 &b\\
    0 &0 &-b &0
\end{pmatrix}\]
where \[a  = \frac{8193-128\sqrt{2}}{8191}, ~ b  = \frac{524289-1024\sqrt{2}}{524287}.\]
As a sanity check, $Z = \ti{F}(\omega)$ indeed satisfies $Z^t = -Z$ and $1-Z\ol{Z}^t > 0$.

\begin{remark}
    In practice, it is hard to determine if the element $\omega$ or its complex conjugate belongs to $\calD^+_T$.
    The sanity check therefore serves as a flag for the potential mistake of choosing $\omega \in \calD^-_T$.
\end{remark}

\section{A rank eighteen specialisation}\label{sect 18}

We will study a specialisation of the same family $\calK_P$ described in Section \ref{sect MAGMA}: the family $\calK_P$ has transcendental lattice $T= U \oplus U(2) \oplus D_4(-1)$ with generators  $\chave{f_1, \cdots, f_4, h_1, \cdots, h_4}$.
Consider the sublattice~$T' = \langle f_1, f_2, f_3, f_4 \rangle = U\oplus U(2)$ of $(T,q)$ which satisfies (\ref{assumption}).
Let $P'$ be its complement in the K3 lattice~$\Lambda_{K3}$.
It can be computed that the lattice $P'$ is given by $U \oplus E_8(-1)\oplus D_8(-1)$.
Then for any~$\omega$ in the identity component $\calD^+_{T'}$ of the period domain of weight two Hodge structures on $T'$, including the example seen at the end of Section \ref{subsect attributes}, $\ti{F}(\omega)$ is in a particularly nice form $Z(a,b)\in M_4(\CC)$
where 
\begin{align*}
    &Z(a,b)_{1,2} = -Z(a,b)_{2,1} = a;\\
    &Z(a,b)_{3,4} = -Z(a,b)_{4,3} = b;\\
    &Z(a,b)_{i,j} = 0 ~\text{if}~ (i,j) \notin \{(1,2), (2,1), (3,4), (4,3)\}.
\end{align*}
Furthermore, the condition  $1-Z\Bar{Z}^t > 0$ tells us that $|a| < 1$ and $|b| < 1$.
This gives an inclusion of the~$2$-dimensional subdomain $\ti{F}(\calD^+_{T'})$ of $\ti{F}(\calD^+_{T})$ into $\calS_1\times \calS_1$, the product of two Siegel upper-half spaces of degree $1$\index{Siegel upper half space}:
\begin{align*}
    \ti{F}(\calD^+_{T'}) &\xhookrightarrow{~~~} D_1 \times D_1 \xrightarrow[~~~]{\simeq}\calS_1\times \calS_1\\
    Z(a,b) &\longmapsto ~~~(a,b) ~~\longmapsto \prnt{f(a), f(b)}
\end{align*}
where $f$ is the conformal map taking a disc $D_1$ to $\calS_1$ by
\[x \longmapsto \frac{\sqrt{-1}(1+x)}{1-x}.\]
Therefore, we may consider $\calD^+_{T'}$ as a subset of the parametrisation space of a pair of elliptic curves.
This observation can be explained by a special geometrical property of the K3 surfaces $X' \in \calK_P$ polarised by $P' \supseteq P$ that associates $X'$ with an abelian surface which can be decomposed into the product of the pair of elliptic curves represented by $\tiF(X')$.
\begin{definition}{\em \cite[Definition 6.1]{mo1}}\label{Def SI}

    A K3 surface $X$ is said to admit a {\em Shioda-Inose structure} associated to an abelian surface $A$ if there is a symplectic involution $\iota$ on $X$ such that the Kummer surface $Y = \Kum(A)$ is the minimal resolution of $X/\langle \iota \rangle$, 
    and if the associated rational double cover $ \pi_X\colon X \dasharrow Y$ induces a Hodge isometry~$(\pi_X)_*\colon T_X(2) \rightarrow T_Y$, where $T_X$ and $T_Y$ are the transcendental lattices of $X$ and $Y$ respectively.
    \end{definition}
There is a lattice theoretic criterion for $X'$ to carry a Shioda-Inose structure \cite[Corollary 2.6]{mo1}: for $X'$ of Picard rank $18$, it is enough to check that its transcendental lattice contains a copy of $U$ as a summand, which is indeed the case here. 

Let us denote by $\KS(X') = \KS(T)$ and $\KS(T')$ the different KS varieties constructed from the weight two Hodge structures on $T$ and $T'$. We will prove the following theorem.

\begin{theorem}\label{thm main}
    Consider the lattices 
    \[P = U \oplus D_8(-1) \oplus D_4(-1) \quad \text{and} \quad P' = U \oplus E_8(-1) \oplus D_8(-1).\]
    Suppose $X' \in \calK_P$ is polarised by $P' \supseteq P$. Let $A_1$ be the image of $X'$ under the modular mapping $F$ constructed in Theorem~\ref{thm F}.
    If~$X'$ is a very general $P'$-polarised K3 surface, that is if $\Pic(X') = P'$, then $A_1$ is isogenous to $E_1^4 \times E_2^4$, where $E_1$ and $E_2$ are two non-isogenous elliptic curves.
\end{theorem}

Before proceeding, let us recall some properties of $\KS(T)$ and $\KS(T')$ due to \cite{mo2}. There, he generalises the Kuga-Satake construction by allowing the input to be any triple $(X, N, \eta)$, where $X$ is a compact complex K\"ahler surface with $\dim H_\CC^{2,0}(X)=1$, $N$ is any subgroup of $H^2(X,\ZZ)$ and $\eta$ is a non-zero element in $H_{4}(X,\ZZ)$.
In particular, when $X$ is a K3 surface, then $\KS(T)$ in our notation coincides with $\KS((X,T^\perp_{H^2(X,\ZZ)} ,[X]))$ in \cite{mo2}.
\begin{lemma}\label{lem mo KS}{\em \cite[Sections 4.4 and 4.7]{mo2}}
    \vspace{-0.5em}
    \begin{enumerate}[(i)]
        \item 
Let $X$ be a K3 surface with transcendental lattice $T$. Let $T', T''$ be lattices such that $T \subset T' \subset T'' \subset H^2(X,\ZZ)$, and let $d = \dim_\QQ((T''/T')\otimes \QQ)$.
Then
\[\KS(T'') \sim \KS(T')^{2^d}.\]

        \item Let $X$ be a K3 surface with  a Shioda-Inose structure \index{Shioda-Inose structure}associated to an abelian surface $A$.
Then
\[\KS(H^2(X,\ZZ)) \sim A^{2^{19}}.\]
    \end{enumerate}
\end{lemma}
The proof and explanation of Lemma \ref{lem mo KS}(ii) in \cite[Section 4.7]{mo2} depend on the statement in part~(i), which is explained in \cite[Remark 2.4]{vv}.

The next step of the proof of Theorem \ref{thm main} involves the following lemma.
\begin{lemma}\label{lem E1E2}
    Suppose $X' \in \calK_P$ is very general.
    Let
    $\KS(X') = \KS(T) \sim A_1 \times \cdots \times A_8$
    be the decomposition of the KS variety described in Section \ref{subsect ex}.
   Then there exist elliptic curves $E_1, E_2$ and an integer $k$ satisfying $0 < k < 8$ such that
    \[A_i \sim E_1^k\times E_2^{8-k} \text{ for all } i = 1, \cdots, 4 \quad \text{ and } \quad A_i \sim E_2^k\times E_1^{8-k} \text{ for all } i = 5, \cdots, 8.\]
\end{lemma}
\begin{proof}
    Let $X'$ be a K3 surface whose transcendental lattice is exactly the rank $4$ lattice $T' \subset T \subset \Lambda_{K3}$.
    By Lemma \ref{lem mo KS}(i),
    we have $\KS(T) \sim \KS(T')^{2^4}$.
    
    On the other hand, recall the K3 surface $X'$ has a Shioda-Inose structure associated to an abelian surface $A'$. 
    From Lemma \ref{lem mo KS}(i), we have $\KS(T')^{2^{18}} \sim \KS(H^2(X',\ZZ))$, 
    and from Lemma \ref{lem mo KS}(ii), we have $\KS(H^2(X',\ZZ)) \sim (A')^{2^{19}}$.
    Futhermore by applying the Poincar\'e's Complete Reducibility Theorem,  we have 
    $\KS(T') \sim (A')^{2}$.

    Finally, from \cite[Table 11.5]{harv}, we have
    $\Cl^+(T'_\RR) \simeq M_2(\RR)^{\oplus 2}$.
    Since assumption (\ref{assumption2}) applies for $T'$,
    Theorem \ref{thm decomposeKS}  implies the decomposition  $\KS(T') \sim (E_1 \times E_2)^2$, where $E_1$ and $E_2$ are non-isogenous elliptic curves.
    Combining all statements,
    this gives $A_i \sim E_1^k\times E_2^{8-k}$.
    Moreover, four subvarieties in the decomposition of $\KS(X')$ described in Theorem \ref{thm decomposeKS} are isogenous to $E_1^k\times E_2^{8-k}$, and the other four are isogenous to $E_1^{8-k}\times E_2^{k}$.
    Note that we can eliminate the cases $k = 0$ or $k = 8$, as the family of $A_i$ is two dimensional.
\end{proof}

\begin{remark}\label{A'}
    Since $A'$ has transcendental lattice $U\oplus U(2)$, its Picard lattice is given by $U(2)$, which suggests that
    \[A' \simeq (E_1\times E_2)/\{(P,Q)\} \sim E_1\times E_2,\]
    where $P\in E_1[2]$ and $Q\in E_2[2]$ are 2-torsion points in the elliptic curves $E_1$ and $E_2$ respectively.
\end{remark}

To prove Theorem \ref{thm main}, it remains to show that $k=4$ in the above statement.
\begin{proof}[Proof of Theorem \ref{thm main}]

Let $(T')^\perp$ be the sublattice in $T$ such that $T = T' \oplus (T')^\perp$. $\ie $ $(T')^\perp = D_4(-1)$.
We recall in Section \ref{subsect ex} that pulling back each pseudo-idempotent $32\varepsilon_i$ along the gluing map 
\[\Cl^+(T') \otimes \Cl^+((T')^\perp) \longrightarrow \Cl^+(T)\]
is the tensor product $x_j \otimes y_k$ for some $j, k$ depending on $i$.
Then by the same reasoning as in the proof of Lemma \ref{lem mo KS}(i) provided in \cite[Remark 2.4]{vv}, we have
\begin{align*}
    \Lambda'_1 \simeq \Cl^+(T)\cdot (32\varepsilon_1) &\simeq \left ( \prnt{\Cl^+(T') \cdot x_j} \otimes \prnt{\Cl^+((T')^\perp)\cdot y_k} \oplus 
    \prnt{\Cl^-(T') \cdot x_j} \otimes \prnt{\Cl^-((T')^\perp)\cdot y_k }\right )\\
    &\simeq 4\prnt{\prnt{\Cl^+(T')\cdot x_j} \oplus \prnt{\Cl^-(T')\cdot x_j}}.
\end{align*}
The second isomorphism comes from the fact that under the algebra isomorphism $\Cl((T'_\RR)^\perp) \rightarrow~M_2(\HH)$, the images of both $\prnt{\Cl{}^+((T')^\perp)\cdot y_k}$ and $\prnt{\Cl{}^-((T')^\perp)\cdot y_k}$ are rank $4$ lattices over $\ZZ$. 

On the other hand, $x_1, \cdots, x_4$ are pseudo-idempotents of $\Cl^+(T')$ by definition.
Similarly by studying the algebra isomorphism $\Cl(T'_\RR) \rightarrow M_2(\RR)^{\oplus 2}$, the lattices $\Cl^+(T')\cdot x_i$ and $\Cl^-(T')\cdot x_i$ are both of rank $2$ over $\ZZ$.
Therefore, they respectively correspond to an elliptic curve $E^+_i$ and $E^-_i$ in the simple decomposition of $\KS(T')$.
This implies $k = 4$.
\end{proof}

Theorem \ref{thm main} implies that $\calD^+_{T'}$ cuts out a special locus in $\calD^+_T$ whose image under $\ti{F}$ corresponds to non-simple abelian $8$-folds which are products in the form of $E_1^4 \times E_2^4$, where $E_1$ and $E_2$ are generically non-isogenous.
Also, we have $A_1 \sim \cdots \sim A_8$. 

\begin{remark}
We can similarly find a $2$-dimensional locus in $\calD^+_T$ for all even, indefinite $2$-elementary transcendental lattices $T$ of rank $8$. 
By \cite{cm2}, such $T$ has a summand of $T' = U\oplus U$, $U\oplus U(2)$ or~$U(2)\oplus U(2)$.
By \cite{mo1} and \cite{mo2}, this implies that $X'$ has a Shioda-Inose structure for the first two cases, or it is a Kummer surface $\Kum(A)$ with $\NS(A) = U$ and
$\KS(X) \sim (A\times A^\vee)^{2^4} \sim A^{2^5}$ in the third case.
In all cases, it is easy to check that all the arguments in the proof of Lemma~\ref{lem E1E2} apply, as they only depend on the rank and the signature of the sublattice $T'$ in $T$.
The proof of Theorem~\ref{thm main} also works nicely: 
if we choose the pseudo-idempotents~$x_1, \cdots, x_4$ in $\Cl(T')$ such that their images under the homomorphism $\Cl(T') \rightarrow M_4(\QQ)$ are some integral multiples of $E_{1,1}$ up to $E_{4,4}$, then  $\Cl^+(T')\cdot x_i$ and $\Cl^-(T')\cdot x_i$ are both of rank $2$ over $\ZZ$ and correspond to two non-isogeneous elliptic curves $E_1$ and $E_2$.
Therefore in both cases, for all~$A_1$ parametrised by $\ti{F}(\calD^+_{T'})$, we  again have the decomposition $A_1 \sim E_1^4 \times E_2^4$.
\end{remark}

\begin{remark}
    An alternative way of proving Theorem~\ref{thm main} suggested by the referee is as follows.
    Take any lattice $T''$ of rank $7$ with $T' < T'' < T$.
    We have, by \cite[Theorem 7.7, (2)]{vg}, \[\Cl^+(T''_\QQ) \simeq M_4(D),\]
    where $D$ is a quaternion algebra over $\QQ$.
    It follows that $\KS(T'')\sim B^4$, where $B$ is an abelian $8$-fold considering that the dimension of $\KS(T'')$ is $32$.
    The abelian variety $B$ is simple if and only if $D \not\cong M_2(\QQ)$, and if $B$ is not simple, then $B = (B')^2$ for a simple abelian $4$-fold $B'$.
    On the other hand, Lemma \ref{lem mo KS}(i) implies $\KS(T'')^2 \sim \KS(T)$, so we have $B^8 \sim A_1 \times A_2 \times \cdots \times A_8$, and so $A_i \sim B$ for any $i$.
    Combined with the result of Lemma~\ref{lem E1E2}, we have $A_1 \sim E_1^4 \times E_2^4$.
\end{remark}

\end{document}